\documentclass[12pt]{amsart}
\usepackage{amssymb}
\usepackage{t1enc}
\usepackage[mathscr]{eucal}
\usepackage{parskip}
\usepackage{xcolor}
\usepackage{tikz}
\usepackage{tikz-cd}
\usepackage{hyperref}
\numberwithin{equation}{section}
\usepackage[margin=2.9cm]{geometry}
\usepackage{booktabs}
\usepackage{enumerate}
\usepackage{bm}
\usepackage{bbm}
\usepackage{xfrac}

\makeatletter
\newcommand{\dashedrightarrow}{%
  \settowidth{\@tempdima}{$\rightarrow$}\rightarrow
  \makebox[-\@tempdima]{\hskip-1.5ex\color{white}\rule[0.5ex]{2pt}{1pt}}
  \phantom{\rightarrow}
}

\theoremstyle{plain}
\newtheorem{thm}{Theorem}[section]
\newtheorem{lem}[thm]{Lemma}
\newtheorem{cor}[thm]{Corollary}
\newtheorem{prop}[thm]{Proposition}
\newtheorem{conj}[thm]{Conjecture}

 \theoremstyle{definition}
\newtheorem{defn}[thm]{Definition}
\newtheorem{rem}[thm]{Remark}
\newtheorem{ex}[thm]{Example}
\newtheorem{notn}[thm]{Notation}

\newcommand{\op}[1]{\operatorname{#1}}
\newcommand{\mb}[1]{\mathbb{#1}}
\newcommand{\mc}[1]{\mathcal{#1}}

\newcommand{\mr}[1]{\mathrm{#1}}
\newcommand{\floor}[1]{\lfloor{#1}\rfloor}
\newcommand{\eps}{\varepsilon}
\newcommand{\vphi}{\varphi}
\newcommand{\Spec}{\operatorname{Spec}}
\newcommand{\GW}{\mathrm{GW}}
\renewcommand{\H}{\mathrm{H}}
\newcommand{\Hdg}{\mathrm{Hdg}}
\newcommand{\Tr}{\mathrm{Tr}}

\newcommand{\1}{\mathbbm{1}}
\newcommand{\SH}{\mathrm{SH}}
\newcommand{\cat}{\mathrm{cat}}
\newcommand{\dR}{\mathrm{dR}}
\newcommand{\mot}{\mathrm{mot}}
\newcommand{\rig}{\mathrm{rig}}
\newcommand{\et}{\mathrm{\acute{e}t}}
\newcommand{\Nis}{\mathrm{Nis}}
\newcommand{\Zar}{\mathrm{Zar}}
\newcommand{\cdh}{\mathrm{cdh}}

\newcommand{\tors}{\mathrm{tors}}
\newcommand{\tr}{\mathrm{tr}}
\newcommand{\id}{\mathrm{id}}

\newcommand{\Var}{\mathrm{Var}}
\newcommand{\Char}{\operatorname{char}}
\newcommand{\End}{\mathrm{End}}
\newcommand{\Hom}{\mathrm{Hom}}
\newcommand{\ev}{\mathrm{ev}}
\newcommand{\dlb}{[\![}
\newcommand{\drb}{]\!]}

\newcommand{\Sym}{\mathrm{Sym}}
\newcommand{\Hilb}{\mathrm{Hilb}}
\newcommand{\Gr}{\mathrm{Gr}}

\newcommand{\uh}{\mathrm{uh}}
\newcommand{\rank}{\operatorname{rank}}
\newcommand{\sign}{\operatorname{sign}}

\makeatletter
\@namedef{subjclassname@2020}{\textup{2020} Mathematics Subject Classification}
\makeatother

\begin{document}
\title{Symmetric powers of null motivic Euler characteristic}

\author{Dori Bejleri}
\address{Department of Mathematics \\ University of Maryland}
\email{dbejleri@umd.edu}
\urladdr{math.umd.edu/~bejleri}

\author{Stephen McKean}
\address{Department of Mathematics \\ Brigham Young University} 
\email{mckean@math.byu.edu}
\urladdr{shmckean.github.io}

\begin{abstract}
    Let $k$ be a field of characteristic not 2. We conjecture that if $X$ is a quasi-projective $k$-variety with trivial motivic Euler characteristic, then $\Sym^n X$ has trivial motivic Euler characteristic for all $n$. Conditional on this conjecture, we show that the Grothendieck--Witt ring admits a power structure that is compatible with the motivic Euler characteristic and the power structure on the Grothendieck ring of varieties. We then discuss how these conditional results would imply an enrichment of G\"ottsche's formula for the Euler characteristics of Hilbert schemes.
\end{abstract}

\maketitle

\section{Introduction}
With an eye towards applications in enumerative geometry over non-closed fields, we investigate a few properties of the motivic Euler characteristic. We begin with a brisk summary of our conjecture and results. We will explain the context of this paper in Section~\ref{sec:motivation}. Let $k$ be a field of characteristic not 2. Let $K_0(\Var_k)$ and $\GW(k)$ denote the Grothendieck ring of varieties over $k$ and the Grothendieck--Witt ring of bilinear forms over $k$, respectively. Let
\[\chi^c:K_0(\Var_k)\to\GW(k)\]
denote the \textit{motivic} or \textit{compactly supported Euler characteristic} (see \cite{Rondigs,AMBOWZ:CompactEuler,LPLS24,Azouri} and Section~\ref{sec:euler char}). Finally, let $\Sym^n:K_0(\Var_k)\to K_0(\Var_k)$ denote the $n\textsuperscript{th}$ symmetric power map, defined on classes of quasi-projective varieties by $[X]\mapsto[\Sym^n X]$. We conjecture that $\Sym^n$ preserves $\ker\chi^c$ for all $n$.

\begin{conj}\label{conj:main}
Let $k$ be a field of characteristic not 2. If $X$ is a quasi-projective $k$-variety such that $[X]\in\ker\chi^c$, then $[\Sym^nX]\in\ker\chi^c$ for all $n\geq 1$.
\end{conj}

The complex and real analogs of Conjecture~\ref{conj:main} are true. Indeed, $X^n\to X^n/S_n$ is \'etale-locally trivial after a suitable stratification, and the rank and signature of $\chi^c$ (which correspond to the compactly supported Euler characteristic of $X(\mb{C})$ and $X(\mb{R})$ by Remark~\ref{rem:rank} and Proposition~\ref{prop:sign = etale of real locus}) factor over \'etale-locally trivial fiber bundles. In particular, if $\chi^c(X)=0$, then the rank and signature\footnote{Some extra care is needed due to the discrepancy between $(\Sym^nX)(\mb{R})$ and $\Sym^n(X(\mb{R}))$; see Corollary~\ref{cor:symmetric torsion} for details.} of $\chi^c(\Sym^n X)$ both vanish for all $n\geq 1$. (This allows us to prove Conjecture~\ref{conj:main} over pythagorean fields; see Theorem~\ref{thm:conj over pythagorean}). In general, $\chi^c$ only factors over locally trivial fiber bundles in topologies coarser than the \'etale topology.

Conjecture~\ref{conj:main} is closely tied to the power structure on $K_0(\Var_k)$ \cite{PowerStructure-K0(Var)}, the motivic Euler characteristic, and a desirable power structure on $\GW(k)$. We make this precise in the following theorem.

\begin{thm}\label{thm:power structure}
Let $k$ be a field of characteristic not 2. Let
\[\mu_0:(1+t\cdot K_0(\Var_k)\dlb t\drb)\times K_0(\Var_k)\to 1+t\cdot K_0(\Var_k)\dlb t\drb\]
be the power structure defined in Section~\ref{sec:power structure K0}. Then Conjecture~\ref{conj:main} is true if and only if there exists a power structure
\[\mu_\GW:(1+t\cdot\GW(k)\dlb t\drb)\times\GW(k)\to 1+t\cdot \GW(k)\dlb t\drb\]
such that
\[\chi^c(\mu_0(A(t),M))=\mu_\GW(\chi^c(A(t)),\chi^c(M))\]
for all $A(t)\in 1+t\cdot K_0(\Var_k)\dlb t\drb$ and $M\in K_0(\Var_k)$.
\end{thm}

In fact, one can explicitly work out a power structure on $\GW(k)$ that matches the conjectural power structure $\mu_\GW$ (see \cite{PP23} for details). The difficulty in Conjecture~\ref{conj:main} is verifying the compatibility of $\mu_0$ and $\mu_\GW$ under the motivic Euler characteristic.

We also prove a few results that do not depend on Conjecture~\ref{conj:main}. For example, we give the following enrichment of G\"ottsche's formula for the Euler characteristic of the Hilbert scheme of points on a surface; this will be stated and proved in the body of the article as Theorem~\ref{thm:unconditional gottsche formula (second)}. An alternative form of this theorem, which depends on Conjecture~\ref{conj:main}, is given in Theorem~\ref{thm:enriched gottsche}.

\begin{thm}\label{thm:unconditional gottsche}
Let $k$ be a field of characteristic not 2. Then for any smooth quasi-projective $k$-surface $X$, we have
\begin{equation*}
\sum_{g\geq 0}\chi^c(\Hilb^g X)\cdot t^g=\prod_{n\geq 1}\left(1+\sum_{m\geq 1}\langle -1\rangle^{m(n-1)}\chi^c(\Sym^m X)\cdot t^{mn}\right)
\end{equation*}
in $\GW(k)\dlb t\drb$.
\end{thm}

Along the way, we collect various properties of $\chi^c$. Many of these already exist in the literature. An example of a new result is that $\chi^c$ is insensitive to universal homeomorphisms (Corollary~\ref{cor:chi universal homeo}). Recently, Pajwani--Rohrbach--Viergever showed that cellular varieties, linear varieties, and elliptic curves satisfy Conjecture~\ref{conj:main} in characteristic 0 \cite{PRV24}. Corollary~\ref{cor:chi universal homeo} implies that their results extend to $\Char{k}\neq 2$.

\subsection{Motivation}\label{sec:motivation}
An influential result in enumerative geometry, formulated by Yau--Zaslow \cite{YZ96} and proved by Beauville \cite{Beauville:Yau-Zaslow}, gives the generating function of certain counts of rational curves on K3 surfaces. Given a projective K3 surface $X$, there exists $g\geq 1$ such that $X$ admits a $g$-dimensional linear system of curves of genus $g$. Denote the (appropriately weighted)\footnote{The weight of a nodal rational curve is $1$ so if all the rational curves in the linear series are nodal, the number is unweighted.} number of rational fibers of this linear system by $n(g)$. Set $n(0):=1$. The \textit{Yau--Zaslow formula} states that
\begin{equation}\label{eq:yau-zaslow}
    \sum_{g\geq 0}n(g)\cdot t^g=\prod_{n\geq 1}(1-t^n)^{-24}.
\end{equation}
As with many counts in algebraic geometry, Equation~\ref{eq:yau-zaslow} is only valid over $\mb{C}$. One can often recover enumerative theorems over the reals by counting with an appropriate sign. Kharlamov--R\u{a}sdeaconu gave a signed analog of the Yau--Zaslow formula over $\mb{R}$ \cite{KR15}. Each real rational curve has a Welschinger invariant $\pm 1$. Given a $g$-dimensional linear system of curves of genus $g$ on a real K3 surface $X$, let $n_+(g)$ and $n_-(g)$ denote the number of rational fibers with Welschinger invariant $+1$ and $-1$, respectively. Set $w(g):=n_+(g)-n_-(g)$. Then the \textit{real Yau--Zaslow formula} states that
\begin{equation}\label{eq:real yau-zaslow}
    \sum_{g\geq 0}w(g)\cdot t^g=\prod_{m\geq 1}(1+t^m)^{-e_\mb{R}}\cdot\prod_{n\geq 1}(1-t^{2n})^{-\frac{e_\mb{C}-e_\mb{R}}{2}},
\end{equation}
where $e_\mb{C}=24$ and $e_\mb{R}$ are the Euler characteristics of $X(\mb{C})$ and $X(\mb{R})$, respectively.

In \textit{enriched, quadratic,} or \textit{$\mb{A}^1$-enumerative geometry}, one applies tools from motivic homotopy theory to enumerative geometry over a wider class of base fields than just $\mb{C}$ or $\mb{R}$. Rather than integer-valued counts, enriched enumerative theorems are equations in the Grothendieck--Witt ring of the base field. Rank and signature of bilinear forms induce ring homomorphisms $\rank:\GW(\mb{C})\to\mb{Z}$ and $\sign:\GW(\mb{R})\to\mb{Z}$. Applying these two maps to an enriched enumerative equation then yields a classical count over $\mb{C}$ and a signed count over $\mb{R}$.

We now expand on the previous paragraph in the context of the Yau--Zaslow formula. Let $k$ be a field, possibly subject to some assumptions. Let $X$ be a K3 surface over $k$ that admits a $g$-dimensional linear system of curves of genus $g$. To each rational curve $C$, we would like to find an isomorphism class of bilinear forms $q(C)\in\GW(k)$ such that $\sign{q(C)}$ is the Welschinger invariant of $C$. Moreover, $\rank{q(C)}$ should recover the counting weight of $C$ over $\mb{C}$ (i.e.~1 if $C$ is nodal). Let $q(g)$ denote the sum of $q(C)$ over the rational fibers in the linear series on $X$. An \textit{enriched Yau--Zaslow formula} should be an equality
\begin{equation}\label{eq:abstract enriched yau-zaslow}
\sum_{g\geq 0}q(g)\cdot t^g=G(t),
\end{equation}
where $G(t)\in\GW(k)\dlb t\drb$ is a power series such that the rank and signature of Equation~\ref{eq:abstract enriched yau-zaslow} recover Equations~\ref{eq:yau-zaslow} and~\ref{eq:real yau-zaslow}, respectively. Applying other field invariants, such as discriminant or Hasse--Witt invariants, to Equation~\ref{eq:abstract enriched yau-zaslow} would then give Yau--Zaslow theorems in other interesting cases, such as over finite fields or number fields.

Pajwani--P\'al gave the first steps towards Equation~\ref{eq:abstract enriched yau-zaslow} \cite{PajwaniPal:YauZaslow}. Their formula is given in characteristic 0 and holds up to rank, signature, and discriminant. Our results streamline various steps in op.~cit., as well as give a conjectural approach to the global side of Yau--Zaslow formulae in positive characteristic. We will not address the local contributions of Yau--Zaslow formulae in this article.

The general roadblock to proving Conjecture~\ref{conj:main} and treating the local terms of the Yau--Zaslow formula over fields other than $\mb{C}$ and $\mb{R}$ is the unwieldy nature of $\chi^c$ over \'etale-locally trivial fibrations. In contrast, compactly supported Euler characteristics factor over \'etale- or analytic-locally trivial fibrations, a fact which is crucial to proving Yau--Zaslow formulae over algebraically closed fields \cite{Beauville:Yau-Zaslow} or $\mb{R}$ \cite{KR15}.

\subsection{Outline}
The first portion of this article is devoted to discussing the motivic Euler characteristic $\chi^c$. We recall the definition of $\chi^c$ over fields of characteristic not 2 in Section~\ref{sec:euler char}. The motivic Euler characteristic was first defined in characteristic 0 in \cite{Rondigs,AMBOWZ:CompactEuler}, and later over perfect fields of characteristic not 2 in \cite{LPLS24,Azouri}. However, we will take a slightly different approach than that of \cite{LPLS24,Azouri}, which will make it easier to study $\chi^c$ as a motivic measure. We will also remove the perfectness assumption on $k$ by applying results of Bachmann--Hoyois \cite{BachmannHoyois} and Elmanto--Khan \cite{ElmantoKhan}.

In Sections~\ref{sec:properties and computations}, \ref{sec:etale with compact support}, and~\ref{sec:modifications}, we collect and prove some properties of the motivic Euler characteristic. We then recall the power structure on $K_0(\Var_k)$ in Section~\ref{sec:power structure K0}, discuss Conjecture~\ref{conj:main} and power structures on $\GW(k)$ in Section~\ref{sec:power structure gw}, and apply our results to G\"ottsche's formula for Hilbert schemes in Section~\ref{sec:gottsche}. Finally, we prove Conjecture~\ref{conj:main} over pythagorean fields in Section~\ref{sec:pythagorean}.

\subsection{Related work}
Some details from this work are present in \cite{PP23}, which was split off from an earlier collaboration. In that article, the authors construct a power structure on $\GW(k)$ that conjecturally matches the one predicted in Theorem~\ref{thm:power structure}. The key difficulty is verifying that this power structure is compatible with the power structure on $K_0(\Var_k)$ and $\chi^c$. They show that these two power structures agree under $\chi^c$ on varieties of dimension 0. In \cite{PRV24}, a larger class of varieties is given on which $\chi^c$ takes the power structure of $K_0(\Var_k)$ to the power structure of $\GW(k)$.

\subsection*{Acknowledgements}
We thank Elden Elmanto, Marc Levine, Davesh Maulik, and Ambrus P\'al for helpful conversations. We thank Jesse Pajwani for helpful conversations, feedback, and for catching a few minor errors in the first version of this paper.

DB and SM each received support from an NSF MSPRF grant (DMS-1803124 and DMS-2202825, respectively). DB is partially supported by DMS-2401483. 

\section{Motivic Euler characteristic}\label{sec:euler char}
In this section, we recall the definition of the \textit{motivic} or \textit{compactly supported Euler characteristic}
\[\chi^c:K_0(\Var_k)\to\GW(k).\]
This definition will hold over fields of characteristic not 2. The starting point is the \textit{categorical Euler characteristic} $\chi^\cat$ in motivic homotopy theory, first defined by Levine \cite{Levine:Euler} (building on the work of Hoyois \cite{Hoyois-quadratic}).

The role of the Euler characteristic in Beauville's approach to the Yau--Zaslow formula cruicially depends on $\chi$ being a motivic measure. The motivic Euler characteristic $\chi^c$ is the result of forcing $\chi^\cat$ to become a motivic measure. In characteristic 0, this can be done explicitly using Bittner's presentation of $K_0(\Var_k)$ \cite{Rondigs,AMBOWZ:CompactEuler}. In characteristic not 2, one can define $\chi^c$ categorically \cite{LPLS24,Azouri}. We will give a slightly different approach than that of \cite{LPLS24,Azouri}, focusing on the nature of $\chi^c$ as a motivic measure. The construction of $\chi^c$ in \cite{LPLS24,Azouri} comes with an assumption that $k$ is perfect; we will remove this assumption.

\subsection{Background}
Let $S$ be a scheme. Let $\SH(S)$ be the stable motivic homotopy category over $S$. The category $\SH(S)$ is symmetric monoidal, where the monoidal structure is given by a smash product $\wedge:=\wedge_S$ and the unit object is the sphere spectrum $\1_S$. In this context, one can speak of strongly dualizable objects (see e.g.~\cite[p.~257]{Levine:Euler}).

\begin{defn}
The \textit{dual} of a motivic spectrum $X\in\SH(S)$ is the object
\[X^\vee:=\Hom_S(X,\1_S).\] 
This gives rise to the evaluation map $\mr{ev}:X^\vee\wedge X\to\1_S$, which induces a map $\mr{can}_A:X^\vee\wedge A\to\Hom_S(X,A)$ for each $A\in\SH(S)$. We say that $X$ is \textit{strongly dualizable} if $\mr{can}_X:X^\vee\wedge X\to\Hom_S(X,X)$ is an isomorphism.
\end{defn}

With this definition in hand (and building on \cite{Hoyois-quadratic}), Levine defines the Euler characteristic of a strongly dualizable motivic spectrum via the categorical definition~\cite[Definition 1.3]{Levine:Euler}.

\begin{defn}
Let $X\in\SH(S)$ be strongly dualizable. The \textit{categorical Euler characteristic} of $X$ is the endomorphism
\[\1_S\to X\wedge X^\vee\to X^\vee\wedge X\to\1_S,\]
denoted $\chi^\cat_S(X)\in\End(\1_S)$. When $X$ is an $S$-variety, we denote $\chi^\cat_S(X):=\chi^\cat_S(\Sigma^\infty_{\mb{P}^1}X_+)$. We will generally omit the base $S$ from the notation unless confusion about the base may arise.
\end{defn}

When $S=\Spec{k}$ is the spectrum of a field, Morel's $\mb{A}^1$-Brouwer degree theorem~\cite[Corollary 1.24]{Morel:A1-AlgebraicTopology} gives an isomorphism $\mr{End}(\1_k)\cong\GW(k)$, so that $\chi^\cat(X)\in\GW(k)$. Under the further assumption that $k$ is perfect of characteristic not 2, Levine--Raksit prove that $\chi^\cat(X)$ can be computed in terms of the de Rham complex for any smooth projective $k$-scheme $X$~\cite{Levine-Raksit}. This was extended to smooth proper schemes over arbitrary fields by \cite[Theorem~1.1]{BachmannWickelgren}

\begin{defn}
Let $k$ be a field. Given a smooth proper $k$-scheme $X$ of dimension $n$, let $\H^i(X;\Omega^j_{X/k})$ denote the Hodge cohomology groups of $X$ for $0\leq i,j\leq n$. By taking 0 maps as differentials, we get a perfect complex
\[\Hdg(X):=\bigoplus_{i,j=0}^n \H^i(X;\Omega^j_{X/k})[j-i]\]
of $k$-vector spaces. Coherent duality defines a trace map $\eta:\H^n(X;\Omega^n_{X/k})\to k$, so we can use the cup product to define perfect pairings
\[\beta_{i,j}:\H^i(X;\Omega^j_{X/k})\otimes\H^{n-i}(X;\Omega^{n-j}_{X/k})\xrightarrow{\smile}\H^n(X;\Omega^n_{X/k})\xrightarrow{\eta}k.\]
The sum $\sum_{i,j=0}^n(-1)^{i+j}\beta_{i,j}$ determines a non-degenerate symmetric bilinear form on $\Hdg(X)$ (see \cite[\S 8D]{Levine-Raksit} or \cite[\S 1]{BachmannWickelgren}). Denote the isomorphism class of this bilinear form by
\[\chi^\dR(X)\in\GW(k).\]
\end{defn}

\begin{thm}[Levine--Raksit, Bachmann--Wickelgren]\label{thm:LevineRaksit}
If $X$ is a smooth proper scheme over a field $k$, then $\chi^\cat(X)=\chi^\dR(X)$.
\end{thm}

Theorem~\ref{thm:LevineRaksit} enables one to compute the categorical Euler characteristic of smooth proper schemes using more classical invariants. Unfortunately, outside of the set of smooth proper schemes, the categorical Euler characteristic fails to satisfy a desirable property: $\chi^\cat$ is not a motivic measure. That is, if $X$ is a variety with closed subvariety $Z$, then $\chi^\cat(X)\neq\chi^\cat(Z)+\chi^\cat(X\backslash Z)$ in general~\cite[Proposition 2.4 (3)]{Levine-quadratic}. However, over fields of characteristic 0, the restriction of $\chi^\cat$ to smooth projective varieties extends to a motivic measure~\cite[Theorem 1.13]{AMBOWZ:CompactEuler}. Moreover, this extension is unique~\cite[Theorem 2.10]{PajwaniPal:YauZaslow}.

\begin{thm}[Arcila-Maya--Bethea--Opie--Wickelgren--Zakharevich, Pajwani--P\'al]
Let $k$ be a field of characteristic 0. There is a unique ring homomorphism
\[\chi^\mot:K_0(\Var_k)\to\GW(k)\]
such that $\chi^\mot([X])=\chi^\cat(X)$ (or equivalently, $\chi^\mot([X])=\chi^\dR(X)$) for all smooth projective irreducible $k$-varieties.
\end{thm}

The \textit{motivic Euler characteristic} $\chi^\mot$ appeared previously in the work of R\"ondigs \cite[Theorem 5.2]{Rondigs}. The assumption that $\Char{k}=0$ allows one to use Bittner's presentation of $K_0(\Var_k)$, which is a crucial aspect of \cite{AMBOWZ:CompactEuler,PajwaniPal:YauZaslow}.

\subsection{Extending the motivic Euler characteristic}
We now extend the definition of $\chi^\mot$ to more general bases by mimicking the categorical definition. We will start by working over a noetherian scheme $S$, which we will eventually assume to be the spectrum of a field. We will also eventually invert the characteristic of the base field whenever resolution of singularities is not known. Our extension of $\chi^\mot$ to fields of characteristic at least 3 recovers Levine--Pepin Lehalleur--Srinivas's and Azouri's generalization of the compactly supported Euler characteristic in this setting \cite[Section 2]{Azouri}, although at times we take a different route towards the desired definition. We also remove their assumption that $k$ is perfect.

To begin, we need to introduce two subcategories of $\SH(S)$. 

\begin{defn}
A motivic spectrum $X\in\SH(S)$ is \textit{compact} if $\Hom_S(X,-)$ commutes with arbitrary direct sums. Let
\begin{enumerate}[(i)]
\item $\SH^\omega(S)$ be the full subcategory of $\SH(S)$ whose objects are compact, and
\item $\SH^\rig(S)$ be the full subcategory of $\SH(S)$ whose objects are strongly dualizable.
\end{enumerate}
\end{defn}

By~\cite{Riou:SpanierWhitehead}, there is an inclusion of subcategories
\begin{equation}\label{eq:Riou}
\SH^\rig(S)\subseteq\SH^\omega(S),
\end{equation}
which is an equality when $S=\Spec{k}$ for $k$ a field admitting resolution of singularities. In general, this inclusion may be strict. For example, $\SH^\rig(S)\neq\SH^\omega(S)$ when $S$ is noetherian of positive dimension~\cite{Cisinski:mathoverflow}.

The key insight that allows us to define the motivic Euler characteristic categorically is that there is a unique realization map from $K_0(\Var_S)$ to the Grothendieck ring of $\SH^\omega(S)$ \cite[Chapter 2, Lemma 3.7.3]{CLNS:MotivicIntegration}. We will state this more precisely in Lemma~\ref{lem:motivic realization}.

\begin{defn}
Let $\mc{C}$ be a triangulated category with a biadditive tensor product $\otimes$. The \textit{Grothendieck group of $\mc{C}$}, denoted $K_0^\triangle(\mc{C})$, is the quotient of free abelian group generated by the set of isomorphism classes of objects in $\mc{C}$ by the set of all objects of the form
\[[X]-[Y]+[Z],\]
where $X\to Y\to Z\to X[1]$ is a distinguished triangle. The Grothendieck group can be promoted to the \textit{Grothendieck ring of $\mc{C}$}, also denoted $K_0^\triangle(\mc{C})$, by imposing the relation $[X\otimes Y]=[X]\cdot[Y]$ for all objects $X,Y$ in $\mc{C}$.
\end{defn}

\begin{rem}
Recall that $\SH(S)$ admits the structure of a triangulated category, where the distinguished triangles are given by (co)fiber sequences. The fiber product $\times_S$ is biadditive with respect to the smash product $\wedge_S$, so we can take the Grothendieck ring of $\SH(S)$. Compact objects in a triangulated category are stable under extensions and retracts, and \cite[Theorem~0.1]{May01} implies the same for strongly dualizable objects. Thus both $\SH^\omega(S)$ and $\SH^\rig(S)$ are thick subcategories of $\SH(S)$, so the triangulated structure on $\SH(S)$ descends to triangulated structures on $\SH^\omega(S)$ and $\SH^\rig(S)$. In particular, we can form $K_0^\triangle(\SH^\omega(S))$ and $K_0^\triangle(\SH^\rig(S))$.
\end{rem}

\begin{lem}[Chambert-Loir--Nicaise--Sebag]\label{lem:motivic realization}
There is a unique ring homomorphism
\[r_\mot:K_0(\Var_S)\to K_0^\triangle(\SH^\omega(S))\]
such that $r_\mot([X])=[q_!\1_X]$ for all quasi-projective varieties $q:X\to S$ (as the classes of such varieties generate $K_0(\Var_S)$ \cite[Chapter 2, Corollary 2.6.6(a)]{CLNS:MotivicIntegration}).
\end{lem}

\begin{rem}
In \cite[Chapter 2, Lemma 3.7.3]{CLNS:MotivicIntegration}, the target of $r_\mot$ is actually $K_0^\triangle(\SH^c(S))$, where $\SH^c(S)\subseteq\SH^\omega(S)$ is the subcategory of \textit{constructible} motivic spectra. However, $\SH(S)$ is constructibly generated (and therefore compactly generated) by \cite[Proposition C.12]{Hoyois-quadratic}, so the constructible objects in $\SH(S)$ are precisely the compact objects \cite[Proposition 1.4.11]{CisinskiDeglise}. It follows that $\SH^c(S)=\SH^\omega(S)$.

Note that for Lemma~\ref{lem:motivic realization} to make sense, we need $q_!\1_X$ to be compact whenever $q:X\to S$ is quasi-projective. Quasi-projective implies separated of finite type \cite[\href{https://stacks.math.columbia.edu/tag/01VX}{Lemma 01VX}]{stacks}, so $q_!$ preserves compact objects \cite[Corollary~4.2.12]{CisinskiDeglise} (the assumption that $S$ is noetherian can be dropped, see e.g.~\cite{Kha21}). Since $\1_X\in\SH(X)$ is compact, it follows that $q_!\1_X$ is compact as well.
\end{rem}

\begin{notn}
Given a set of primes $\mc{P}$, let $\SH(S)[\mc{P}^{-1}]$ denote the localization of $\SH(S)$ at the maps $p:E\to E$ for all $E\in\SH(S)$ and $p\in\mc{P}$. If $\mc{P}=\{p\}$, then we will write $\SH(S)[p^{-1}]:=\SH(S)[\mc{P}^{-1}]$. We denote the localization functor by
\[L_\mc{P}:\SH(S)\to\SH(S)[\mc{P}^{-1}].\]
Since localization simply introduces more morphisms to our category, we can restrict to compact or strongly dualizable objects to obtain localization maps $L_\mc{P}:\SH^\bullet(S)\to\SH^\bullet(S)[\mc{P}^{-1}]$, where $\bullet=\omega$ or $\rig$.
\end{notn}

\begin{cor}
For any set of primes $\mc{P}$, there is a unique ring homomorphism
\[r_\mot:K_0(\Var_S)\to K_0^\triangle(\SH^\omega(S)[\mc{P}^{-1}])\]
such that $r_\mot([X])=[q_!\1_X]$ for all quasi-projective varieties $q:X\to S$.
\end{cor}
\begin{proof}
It suffices to construct a ring homomorphism 
\[L_\mc{P}:K_0^\triangle(\SH^\omega(S))\to K_0^\triangle(\SH^\omega(S)[\mc{P}^{-1}])\]
taking the desired values on quasi-projective varieties. On objects, this map is given by $L_\mc{P}([X])=[L_\mc{P}X]$. This is well-defined, as two isomorphic motivic spectra remain isomorphic after localization. Localization is an exact functor, which implies that $L_\mc{P}$ preserves:
\begin{itemize}
    \item exact triangles, so $L_\mc{P}X\to L_\mc{P}Y\to L_\mc{P}Z$ is a (co)fiber sequence in $\SH^\omega(S)[\mc{P}^{-1}]$ for any (co)fiber sequence $X\to Y\to Z$ in $\SH^\omega(S)$, and 
    \item finite limits, so that $L_\mc{P}(X\times_S Y)\cong L_\mc{P}X\times_{L_\mc{P}S}L_\mc{P}Y$. 
\end{itemize}
It follows that $L_\mc{P}$ is a ring homomorphism. The localization $\SH^\omega(S)\to\SH^\omega(S)[\mc{P}^{-1}]$ is unique up to unique isomorphism by the universal property of localizations, so the induced map $L_\mc{P}$ on Grothendieck rings is uniquely determined. It now follows from Lemma~\ref{lem:motivic realization} that
\[K_0(\Var_S)\xrightarrow{r_\mot}K_0^\triangle(\SH^\omega(S))\xrightarrow{L_\mc{P}}K_0^\triangle(\SH^\omega(S)[\mc{P}^{-1}])\]
is uniquely determined. To simplify notation, we will generally conflate $r_\mot$ and $L_\mc{P}\circ r_\mot$.
\end{proof}

Our next goal is to show that there is a unique Euler characteristic on $K_0^\triangle(\SH^\omega(S))$ that satisfies the expected categorical definition on $K_0^\triangle(\SH^\rig(S))$. We will begin by constructing a ring homomorphism $K_0^\triangle(\SH^\rig(S))\to\End(\1_S)$ (as well as localizations thereof). We will then conclude with results of Riou and Elmanto--Khan~\cite{Riou:SpanierWhitehead,ElmantoKhan}, which equate certain localizations of $\SH^\omega(k)$ and $\SH^\rig(k)$.

\begin{lem}\label{lem:rigid Euler}
Let $\mc{P}$ be a (possibly empty) set of primes. Then the map 
\[\chi^\rig:K_0^\triangle(\SH^\rig(S)[\mc{P}^{-1}])\to\End_{\SH(S)}(\1_S)[\mc{P}^{-1}]\]
defined by 
\[\chi^\rig([X])=[\1_S\to X\wedge X^\vee\to X^\vee\wedge X\to\1_S]\]
is a ring homomorphism.
\end{lem}
\begin{proof}
Note that $\End_{\SH(S)[\mc{P}^{-1}]}(\1_S)\cong\End_{\SH(S)}(\1_S)[\mc{P}^{-1}]$, as the only additional endomorphisms in the localized category are precisely the maps $p:\1_S\to\1_S$. To begin, we show that $\chi^\rig$ is well-defined. If $\phi:X\to Y$ is an isomorphism in $\SH^\rig(S)[\mc{P}^{-1}]$ (or even in $\SH(S)[\mc{P}^{-1}]$), we get an induced isomorphism
\[\phi^*:X^\vee:=\Hom_S(X,\1_S)\xrightarrow{-\circ\phi^{-1}}\Hom_S(Y,\1_S)=:Y^\vee.\]
It follows that the diagram
\[\begin{tikzcd}[row sep=0.75em]
        & X\wedge X^\vee\arrow[r,"\tau"]\arrow[dd,"\phi\wedge\phi^*"] & X^\vee\wedge X\arrow[dd,"\phi^*\wedge\phi"']\arrow[dr,"\ev_X"] & \\
   \1_S\arrow[ur,"\mr{coev}_X"]\arrow[dr,"\mr{coev}_Y"'] & & & \1_S \\
        & Y\wedge Y^\vee\arrow[r,"\tau"] & Y^\vee\wedge Y\arrow[ur,"\ev_Y"'] &
\end{tikzcd}\]
commutes, where $\tau$ is the swap map (i.e.~tausch map) and $\mr{coev}$ denotes coevaluation. In particular, $\chi^\rig([X])$ is independent of our choice of representative of $[X]$, so the map $\chi^\rig$ is well-defined.

It remains to show that $\chi^\rig$ is a ring homomorphism. Since $\1_S^\vee\cong\1_S$, we find that $\chi^\rig([\1_S])=\id_{\1_S}$, which is the identity element of the ring $\End_{\SH(S)}(\1_S)[\mc{P}^{-1}]$. The additivity of $\chi^\rig$ over (co)fiber sequences follows from \cite{May01}. Note that $\chi^\rig([X]\cdot[Y])=\chi^\rig([X\wedge Y])$ by definition, so multiplicativity of $\chi^\rig$ follows from the multiplicativity of the Euler characteristic in any symmetric monoidal category.
\end{proof}

Our next goal is to extend this \textit{rigid Euler characteristic} $\chi^\rig$ to a ring homomorphism $K_0^\triangle(\SH^\omega(S)[\mc{P}^{-1}])\to\End(\1_S)[\mc{P}^{-1}]$ when $S=\Spec{k}$ and $\mc{P}=\{\exp(k)\}$.

\begin{lem}\label{lem:compact Euler}
Let $k$ be a field of exponential characteristic $e$. Then there is a unique ring homomorphism
\[\chi^\omega:K_0^\triangle(\SH^\omega(k)[e^{-1}])\to\GW(k)[e^{-1}]\]
such that $\chi^\omega([X])=\chi^\rig([X])$ for all $X\in\SH^\rig(S)$.
\end{lem}
\begin{proof}
By \cite[Theorem 10.12]{BachmannHoyois}, $\End_{\SH(k)}(\1_k)\cong\GW(k)$ for any field $k$ (generalizing Morel's result for perfect fields \cite[Corollary 1.24]{Morel:A1-AlgebraicTopology}).  The lemma now follows from the fact that $\SH^\rig(k)[e^{-1}]=\SH^\omega(k)[e^{-1}]$ when $k$ is a perfect field~\cite{Riou:SpanierWhitehead} or in fact any field \cite[Theorem 3.2.1]{ElmantoKhan}.
\end{proof}

\begin{rem}
Riou further proved that $\SH^\omega(k)=\SH^\rig(k)$ for any field $k$ admitting a weak form of resolution of singularities \cite{Riou:SpanierWhitehead}. For a more detailed account of this proof, see \cite[Theorem 52]{RondigsOstvaer:modules}. It follows that if resolution of singularities holds over $k$, then Lemma~\ref{lem:compact Euler} remains true even without inverting the exponential characteristic.
\end{rem}

We can now define a $\GW(k)[e^{-1}]$-valued Euler characteristic that is a motivic measure over $k$.

\begin{defn}\label{def:chi^c}
Let $k$ be a field of exponential characteristic $e$. Define the \textit{compactly supported Euler characteristic} over $k$ as the composite 
\[\chi^c:=\chi^\omega\circ r_\mot:K_0(\Var_k)\to\GW(k)[e^{-1}].\]
\end{defn}

\begin{rem}\label{rem:chi^c valued in gw}
We will later see that $\rank\chi^c$ is always an integer (Proposition~\ref{prop:rank = etale}). If $\Char{k}\neq 2$, it will follow that the image of $\chi^c$ is always contained in $\GW(k)\subseteq\GW(k)[e^{-1}]$ (Corollary~\ref{cor:gw inside gw[1/p]}). In light of Corollary~\ref{cor:gw inside gw[1/p]}, we may conflate notation and write 
\[\chi^c:K_0(\Var_k)\to\GW(k)\]
over any field of odd exponential characteristic. For now, we will continue to specify that $\chi^c$ lands inside $\GW(k)[e^{-1}]$. However, starting in Section~\ref{sec:power structure gw}, it will become necessary for us to restrict our study to fields of odd exponential characteristic and think of $\chi^c$ as taking values in $\GW(k)$.
\end{rem}

Since we know the value of $r_\mot$ and $\chi^\omega$ on quasi-projective varieties, we get a nice formula for $\chi^c$ on quasi-projective varieties.

\begin{prop}\label{prop:chi on quasi proj}
The compactly supported Euler characteristic satisfies
\[\chi^c([X])=[\1_k\to q_!\1_X\wedge(q_!\1_X)^\vee\to(q_!\1_X)^\vee\wedge q_!\1_X\to\1_k]\]
for all quasi-projective $q:X\to\Spec{k}$.
\end{prop}
\begin{proof}
If $q:X\to\Spec{k}$ is quasi-projective, then $r_\mot([X])=[q_!\1_X]$ by Lemma~\ref{lem:motivic realization}. After localizing at $\exp(k)$, we have $q_!\1_X\in\SH^\omega(k)[e^{-1}]=\SH^\rig(k)[e^{-1}]$, so Lemmas~\ref{lem:rigid Euler} and~\ref{lem:compact Euler} give the desired result.
\end{proof}

\begin{rem}
By construction, $\chi^c$ is uniquely determined by its image on quasi-projective varieties, which implies that our construction coincides with that of \cite[\S 2]{Azouri}.
\end{rem}

When $q:X\to\Spec{k}$ is proper, we have a natural isomorphism $q_!\cong q_*$. It follows that $\chi^c=\chi^\cat[e^{-1}]$ on smooth proper schemes and hence $\chi^c=\chi^\dR[e^{-1}]$ on smooth proper varieties over fields.

\begin{prop}\label{prop:compact=cat}
If $q:X\to\Spec{k}$ is smooth and proper, then $\chi^c([X])=\chi^\cat(X)[e^{-1}]$.
\end{prop}
\begin{proof}
If $q:X\to\Spec{k}$ is smooth and proper, then $q_!\1_X:= q_!q^*\1_k$ is strongly dualizable with dual $q_*q^!\1_k\simeq\Sigma^\infty_{\mb{P}^1}X_+\in\SH(k)$ (see e.g.~\cite[\S 3]{Hoyois-quadratic}). Thus $\chi^c([X])=\chi^\cat(X^\vee)[e^{-1}]$ on (infinite suspensions of) smooth proper $k$-schemes, and we conclude by noting that $\chi^\cat(X^\vee)=\chi^\cat(X)$.
\end{proof}

This gives us the desired connection to the motivic Euler characteristic over fields of characteristic 0.

\begin{cor}\label{cor:compact=motivic}
Let $k$ be a field of characteristic 0. Then $\chi^c=\chi^\mot$.
\end{cor}
\begin{proof}
By~\cite[Theorem 2.10]{PajwaniPal:YauZaslow},
\[\chi^\mot:K_0(\Var_k)\to\GW(k)\] 
is the unique ring homomorphism such that $\chi^\mot([X])=\chi^\dR(X)$ for all smooth projective irreducible varieties $X$. But any such variety is smooth and proper, so $\chi^c([X])=\chi^\cat(X)=\chi^\dR(X)$ by Proposition~\ref{prop:compact=cat} and Theorem~\ref{thm:LevineRaksit}. Thus
\[\chi^c:K_0(\Var_k)\to\End(\1_k)\cong\GW(k)\]
is a ring homomorphism such that $\chi^c([X])=\chi^\dR(X)$ for all smooth projective irreducible varieties $X$, so $\chi^c=\chi^\mot$.
\end{proof}

\begin{rem}
Recall that the motivic Euler characteristic $\chi^\mot$ was called the compactly supported Euler characteristic in~\cite{AMBOWZ:CompactEuler}. Corollary~\ref{cor:compact=motivic} justifies calling $\chi^c$ the compactly supported Euler characteristic. In Section~\ref{sec:etale with compact support}, we will see further justification for calling $\chi^c$ the compactly supported Euler characteristic.

Alternatively, one could call $\chi^c$ the \textit{motivic Euler characteristic}, now defined over a more general base than just fields of characteristic 0. The name \textit{motivic} Euler characteristic is justified both by $\chi^c$ being a motivic measure (i.e.~a homomorphism out of $K_0(\Var_k)$) and by its construction via motivic homotopy theory.

In favor of brevity, we will generally use the term \textit{motivic} rather than \textit{compactly supported}.
\end{rem}

\section{Properties and computations of the motivic Euler characteristic}\label{sec:properties and computations}
We will now prove various properties of and formulas for $\chi^c$, some of which were shown in characteristic 0 in \cite{PajwaniPal:YauZaslow}. Most of the formulas in Sections~\ref{sec:basic properties}, \ref{sec:computations}, \ref{sec:bundles}, and~\ref{sec:etale with compact support} are either inspired by computations of $\chi^\cat$ by Hoyois~\cite{Hoyois-quadratic} and Levine~\cite{Levine-quadratic} or are direct consequences of $\chi^c$ being a motivic measure. In many cases (e.g.~whenever all schemes under consideration are smooth and proper), these formulas follow directly from Hoyois and Levine. Regardless, we include the proofs in each case for the reader's convenience.

\subsection{Basic properties}\label{sec:basic properties}
We begin by recording a few instances where $\chi^c$ is additive or multiplicative. Because $\chi^c:K_0(\Var_k)\to\GW(k)[e^{-1}]$ is a ring homomorphism, the motivic Euler characteristic inherits nice properties with respect to the additive and multiplicative structure in $K_0(\Var_k)$.

\begin{prop}\label{prop:chi additive}
Let $X$ be a $k$-scheme, with $U\subset X$ an open subscheme. Then $\chi^c(X)=\chi^c(U)+\chi^c(X-U)$.
\end{prop}
\begin{proof}
We have the identity $[X]=[U]+[X-U]$ in $K_0(\Var_k)$, and $\chi^c$ respects the additive structure of $K_0(\Var_k)$.
\end{proof}

\begin{prop}\label{prop:chi mult}
Let $X$ and $Y$ be $k$-schemes. Then $\chi^c(X\times_k Y)=\chi^c(X)\cdot\chi^c(Y)$.
\end{prop}
\begin{proof}
We have the identity $[X\times_k Y]=[X]\cdot[Y]$ in $K_0(\Var_k)$, and $\chi^c$ respects the mutliplicative structure of $K_0(\Var_k)$.
\end{proof}

As a consequence, we find that $\chi^c$ is multiplicative over any Zariski-locally trivial fiber bundle.

\begin{cor}\label{cor:chi mult fiber bundles}
Let $X$ be a connected $k$-scheme. Let $p:Y\to X$ be a Zariski-locally trivial fiber bundle with fiber $F$. Then $\chi^c(Y)=\chi^c(X)\cdot\chi^c(F)$.
\end{cor}
\begin{proof}
This follows from the well-known identity $[Y]=[X]\cdot[F]$ in $K_0(\Var_k)$. To verify this identity, we take a locally closed stratification $X=\coprod_i U_i$ such that each $p^{-1}(U_i)\to U_i$ is a trivial fiber bundle. By Proposition~\ref{prop:chi mult}, we have $[p^{-1}(U_i)]=[U_i]\cdot[F]$. Now
\begin{align*}
    [Y]&=\sum_i[p^{-1}(U_i)]\\
    &=[F]\sum_i[U_i]\\
    &=[F]\cdot[X].\qedhere
\end{align*}
\end{proof}

Finally, we will prove a few propositions showing that $\chi^c$ is compatible with field extensions. If $L/k$ is a field extension, let $\chi^c_k$ and $\chi^c_L$ denote the motivic Euler characteristic over $k$ and $L$, respectively. 

\begin{prop}\label{prop:field_ext}
Let $i:k\to L$ be a field extension, and let $i_*:\GW(k)[e^{-1}]\to\GW(L)[e^{-1}]$ be the induced homomorphism on Grothendieck--Witt rings. Given $X\in\Var_k$, denote the base change $X_L:=X\times_{\Spec{k}}\Spec{L}\in\Var_L$. Then $i_*\chi^c(X)=\chi^c(X_L)$.
\end{prop}
\begin{proof}
Let $\pi:\Spec{L}\to\Spec{k}$ be the map of schemes induced by $i$. Then we have the pullback functor $\pi^*:\SH(k)[e^{-1}]\to\SH(L)[e^{-1}]$, which is exact and symmetric monoidal. Note that $\End_{\SH(L)[e^{-1}]}(\1_L)\cong\End_{\SH(L)}(\1_L)[e^{-1}]$, as $e:\1_L\to\1_L$ is the only additional endomorphism introduced by localization. We thus have $\pi^*:\End_{\SH(k)[e^{-1}]}(\1_k)\to\End_{\SH(L)[e^{-1}]}(\1_L)$, which is equivalent to $i_*:\GW(k)[e^{-1}]\to\GW(L)[e^{-1}]$ by Morel's isomorphism $\End_{\SH(F)}(\1_F)\cong\GW(F)$.

Let $p:X\to\Spec{k}$ and $q:X_L\to\Spec{L}$ be the structure maps, fitting inside the cartesian square
\[\begin{tikzcd}
X_L\arrow[r,"f"]\arrow[d,"q"'] & X\arrow[d,"p"]\\
\Spec{L}\arrow[r,"\pi"] & \Spec{k}.
\end{tikzcd}\]
The exchange morphism $\pi^*p_!\to q_!f^*$ is an isomorphism by \cite[Proposition~2.2.14]{CisinskiDeglise}, since proper base change holds in the stable motivic homotopy category and $\pi$ is separated (being a morphism between affine schemes). Thus the monoidality of $f^*$ implies that $\pi^*p_!\1_X\cong q_!\1_{X_L}$. In particular, $i_*\chi^c_k(X)=\chi^c_L(X_L)$ whenever $p:X\to\Spec{k}$ is quasi-projective. We thus have two ring homomorphisms
\[\chi^c_L(-_L),i_*\chi^c_k:K_0(\Var_k)\to\GW(L)[e^{-1}]\]
that coincide on quasi-projective varieties. It follows from \cite[Chapter 2, Corollary 2.6.6]{CLNS:MotivicIntegration} that $\chi^c_L(-_L)=i_*\chi^c_k$.
\end{proof}

Given a field extension $L/k$, we get an induced structure map $p:\Spec{L}\to\Spec{k}$. Given an $L$-variety $q:X\to\Spec{L}$, we can thus view $X$ as a $k$-variety $p\circ q:X\to\Spec{k}$. This induces a group (but not ring) homomorphism $\pi_{L/k}:K_0(\Var_L)\to K_0(\Var_k)$.

\begin{prop}\label{prop:trace}
Let $L/k$ be a finite separable field extension. Then the following diagram of groups commutes:
\[\begin{tikzcd}
K_0(\Var_L)\arrow[r,"\pi_{L/k}"]\arrow[d,"\chi^c_L"'] & K_0(\Var_k)\arrow[d,"\chi^c_k"]\\
\GW(L)[e^{-1}]\arrow[r,"\Tr_{L/k}"] & \GW(k)[e^{-1}].
\end{tikzcd}\]
\end{prop}
\begin{proof}
Let $p:\Spec{L}\to\Spec{k}$ be the structure map, which is \'etale by our separability assumption. We have a left adjoint $p_\#:\SH(L)\to\SH(k)$ to $p^*$. Given an endomorphism $\omega\in\End(\1_L)$, \cite[Proposition 5.2]{Hoyois-quadratic} states that $\tr(p_\#\omega)=\Tr_{L/k}\omega$, where $\tr$ denotes the categorical trace. Setting $\omega=\chi^c_L$, we have
\begin{align}\label{eq:trace down}
    \Tr_{L/k}\chi^c_L(X)&=[\1_k\to p_\#q_!\1_X\wedge(p_\#q_!\1_X)^\vee\to(p_\#q_!\1_X)^\vee\wedge p_\#q_!\1_X\to\1_k]
\end{align}
for every quasi-projective $q:X\to\Spec{L}$. Since $p$ is \'etale, we have a natural isomorphism $p_\#\simeq p_!$, so $p_\#q_!\1_X\cong(p\circ q)_!\1_X$. Equation~\ref{eq:trace down} thus simplifies to
\begin{align*}
\Tr_{L/k}\chi^c_L(X)&=[\1_k\to(p\circ q)_!\1_X\wedge((p\circ q)_!\1_X)^\vee\to((p\circ q)_!\1_X)^\vee\wedge(p\circ q)_!\1_X\to\1_k],
\end{align*}
which is equal to $\chi^c_k\pi_{L/k}(X)$. We thus have two group homomorphisms
\[\Tr_{L/k}\chi^c_L,\chi^c_k\pi_{L/k}:K_0(\Var_L)\to\GW(k)[e^{-1}]\] that coincide on quasi-projective varieties. It follows from \cite[Chapter 2, Corollary 2.6.6.]{CLNS:MotivicIntegration} that $\Tr_{L/k}\chi^c_L=\chi^c_k\pi_{L/k}$ as group homomorphisms.
\end{proof}

\subsection{Computations for various schemes}\label{sec:computations}
We now compute the motivic Euler characteristic of points, projective spaces, affine spaces, and a few other interesting schemes.

\begin{notn}
Given a positive integer $n$, let
\[n_\eps:=\underbrace{\langle 1\rangle+\langle -1\rangle+\langle 1\rangle+\cdots+\langle -1\rangle^{n-1}}_{n\text{ times}},\]
which can be considered as an element of $\GW(k)$ or $\GW(k)[e^{-1}]$. We also define $0_\eps:=0$.
\end{notn}

\begin{prop}\label{prop:chi of P^n}
We have $\chi^c(\Spec{k})=\langle 1\rangle$ and $\chi^c(\mb{P}^n_k)=(n+1)_\eps$.
\end{prop}
\begin{proof}
Since $\Spec{k}$ and $\mb{P}^n_k$ are smooth and proper over $k$, it follows that $\chi^c(X)=\chi^\cat(X)[e^{-1}]$ when $X$ is one of these schemes. The result now follows from \cite[Example 1.7]{Hoyois-quadratic}.
\end{proof}

As a corollary, we can compute the motivic Euler characteristic of affine space.

\begin{cor}\label{cor:chi A^n}
We have $\chi^c(\mb{A}^n_k)=\langle -1\rangle^n$.
\end{cor}
\begin{proof}
We use the decomposition $[\mb{P}^n_k]=[\mb{A}^n_k]+[\mb{P}^{n-1}_k]$ in $K_0(\Var_k)$. By Propositions~\ref{prop:chi additive} and~\ref{prop:chi of P^n}, we thus have
\begin{align*}
    \chi^c(\mb{A}^n_k)&=\chi^c(\mb{P}^n_k)-\chi^c(\mb{P}^{n-1}_k)\\
    &=(n+1)_\eps-n_\eps\\
    &=\langle-1\rangle^n.
\end{align*}
Alternatively, one can compute $\chi^c(\mb{A}^1_k)=\langle-1\rangle$ in this manner and then apply Proposition~\ref{prop:chi mult} to $\mb{A}^n_k=\prod_{i=1}^n\mb{A}^1_k$.
\end{proof}

We can now obtain various important classes in $\GW(k)[e^{-1}]$ as the motivic Euler characteristic of relatively simple schemes.

\begin{cor}\label{cor:punctured affine}
Let $p_1,\ldots,p_r\in\mb{A}^n_k(k)$ be distinct $k$-rational points. Then $\chi^c(\mb{A}^n_k-\{p_1,\ldots,p_r\})=\langle -1\rangle^n-r\langle 1\rangle$. In particular,
\begin{enumerate}[(i)]
\item $\chi^c(\mb{G}_m)=\langle-1\rangle-\langle 1\rangle$,
\item $\chi^c(\mb{P}^1_k-\{0,1,\infty\})=\langle-1\rangle-2\langle 1\rangle$,
\item $\chi^c(\mb{A}^{2m}_k-\{0\})=0$, and
\item $\chi^c(\mb{A}^{2m}_k-\{0,1\})=-\langle 1\rangle$.
\end{enumerate}
\end{cor}
\begin{proof}
Since $p_1,\ldots,p_r$ are distinct, this follows immediately from Proposition~\ref{prop:chi additive} after computing $\chi^c(\mb{A}^n_k)$ (Corollary~\ref{cor:chi A^n}) and $\chi^c(\Spec{k})$ (Proposition~\ref{prop:chi of P^n}).
\end{proof}

\begin{ex}
We will later prove that the signature of $\chi^c(X)$ corresponds to the compactly supported Euler characteristic of the real locus $X(\mb{R})$ (see Proposition~\ref{prop:sign = etale of real locus}). The real locus of $\mb{G}_m$ is the affine hyperbola illustrated in Figure~\ref{fig:hyperbola}, which has Euler characteristic $-2=\sign(\langle-1\rangle-\langle 1\rangle)$.
\end{ex}

\begin{figure}
\begin{tikzpicture}[scale=0.75]
  \draw[->] (-3, 0) -- (3,0);
  \draw[->] (0, -3) -- (0, 3);
  \draw[ultra thick,domain=1/3:3,smooth,variable=\x,red] plot ({\x}, {1/\x});
  \draw[ultra thick,domain=-3:-1/3,smooth,variable=\x,red] plot ({\x}, {1/\x});
\end{tikzpicture}
\caption{$\mb{G}_m(\mb{R})$ with compactly supported Euler characteristic $-2$}\label{fig:hyperbola}
\end{figure}
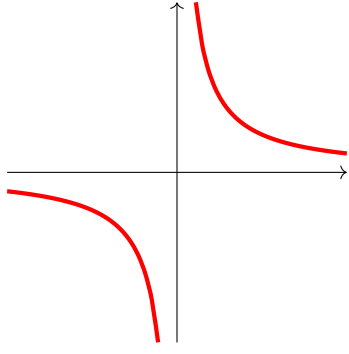

\begin{ex}\label{ex:motivic = cat}
By Proposition~\ref{prop:compact=cat}, we can compute the motivic Euler characteristic of any smooth proper scheme by computing its categorical Euler characteristic. Here, we record a few more computations from the literature.
\begin{enumerate}[(i)]
\item If $L/k$ is a finite separable field extension, then $\chi^c(\Spec{L})=\Tr_{L/k}\langle 1\rangle$, where $\Tr_{L/k}:\GW(L)[e^{-1}]\to\GW(k)[e^{-1}]$ is induced by the field trace \cite[Theorem 1.9]{Hoyois-quadratic}.
\item If $X$ is a Brauer--Severi variety of dimension $n$ over $k$, then $\chi^c(X)=(n+1)_\eps$ \cite[Examples 2.6]{Levine-quadratic}. Note that we generally do not have an equality of $[X]$ and $[\mb{P}^n]$ in $K_0(\Var_k)$ \cite[Theorem 7]{Lit15}, so the computation $\chi^c(X)=\chi^c(\mb{P}^n)$ shows that the class $[X]-[\mb{P}^n]$ lies in $\ker\chi^c$. Since $\chi^c$ is a ring homomorphism, we also find that $\chi^c(P)=1$, where $P\in K_0(\Var_k)$ is the class mentioned in \cite[Theorem 7]{Lit15}.
\item Let $\Gr_k(r,n)$ be the Grassmannian of $r$-planes in $k^n$. Set $n_\mb{C}:=\binom{n}{r}$ and $n_\mb{R}:=\binom{\floor{n/2}}{\floor{r/2}}$. Then $\chi^c(\Gr_k(r,n))=\frac{n_\mb{C}+n_\mb{R}}{2}\langle 1\rangle+\frac{n_\mb{C}-n_\mb{R}}{2}\langle-1\rangle$. Indeed, since Grassmannians are smooth and projective over $\mb{Z}$, this follows from \cite[Theorem 5.11]{BachmannWickelgren} and the computations $\rank\chi^\cat(\Gr_k(r,n))=\chi(\Gr_\mb{C}(r,n))=n_\mb{C}$ and $\sign\chi^\cat(\Gr_k(r,n))=\chi(\Gr_\mb{R}(r,n))=n_\mb{R}$. Alternatively, one can compute $\chi^\cat(\Gr_k(r,n))$ using the cellular structure of $\Gr_k(r,n)$ (see e.g.~\cite[Proposition 8.3 and Theorem 8.4]{BrazeltonMcKeanPauli}).
\end{enumerate}
\end{ex}

\begin{rem}
At this point, we have given enough examples to prove that $\chi^c$ surjects onto $\GW(k)$ when $\Char{k}\neq 2$. This fact, which we prove in Proposition~\ref{prop:chi^c surjective}, will be essential to the broader goals of the article, so we postpone its proof until the surjectivity becomes pertinent.
\end{rem}

\subsection{Computations for bundles}\label{sec:bundles}
Thanks to Corollary~\ref{cor:chi mult fiber bundles}, we can compute the motivic Euler characteristic of vector bundles, projective bundles, and blow ups.

\begin{cor}\label{cor:chi vector bundle}
Let $V\to X$ be a vector bundle of rank $r$. Then $\chi^c(V)=\langle-1\rangle^r\cdot\chi^c(X)$.
\end{cor}
\begin{proof}
A vector bundle of rank $r$ is Zariski-locally trivial with fiber $\mb{A}^r_k$, so the result follows from Corollaries~\ref{cor:chi mult fiber bundles} and~\ref{cor:chi A^n}.
\end{proof}

We also get a direct computation of the motivic Euler characteristic of projective bundles.

\begin{cor}\label{cor:chi projective bundle}
Let $V\to X$ be a vector bundle of rank $r$. Then $\chi^c(\mb{P}V)=r_\eps\cdot\chi^c(X)$.
\end{cor}
\begin{proof}
The projective bundle $\mb{P}V$ is a Zariski-locally trivial fiber bundle with fibers $\mb{P}^{r-1}_k$, so the result follows from Corollary~\ref{cor:chi mult fiber bundles} and Proposition~\ref{prop:chi of P^n}.
\end{proof}

Since the exceptional divisor of a blow up along a smooth subscheme is a projective bundle, the motivic Euler characteristic satisfies a fairly simple formula in this setting:

\begin{prop}\label{prop:chi blow up}
Let $X$ be a $k$-scheme. Let $Z\subset X$ be an lci closed subscheme of codimension $d$. Let $\tilde{X}$ denote the blow up of $X$ along $Z$. Then
\[\chi^c(\tilde{X})=\chi^c(X)+\langle-1\rangle\cdot(d-1)_\eps\cdot\chi^c(Z).\]
In particular, if $Z$ is of codimension 1, then $\chi^c(\tilde{X})=\chi^c(X)$.\footnote{If $Z$ is reduced of codimension 1, then $\tilde{X}\cong X$, which directly implies $\chi^c(\tilde{X})=\chi^c(X)$.}
\end{prop}
\begin{proof}
Let $N_{Z/X}\to Z$ be the normal bundle of $Z$ inside $X$. Let $E=\mb{P}N_{Z/X}\to Z$ be the exceptional divisor of $\tilde{X}$, so that $X-Z\cong\tilde{X}-E$. Since $N_{Z/X}$ has rank $d$, Proposition~\ref{prop:chi additive} and Corollary~\ref{cor:chi projective bundle} thus imply
\begin{align*}
    \chi^c(X)-\chi^c(Z)&=\chi^c(\tilde{X})-\chi^c(E)\\&=\chi^c(\tilde{X})-d_\eps\cdot\chi^c(Z).
\end{align*}
We conclude by noting that
\begin{align*}
d_\eps\cdot\chi^c(Z)-\chi^c(Z)&=(d_\eps-\langle 1\rangle)\cdot\chi^c(Z)\\
&=\langle-1\rangle\cdot(d-1)_\eps\cdot\chi^c(Z).\qedhere
\end{align*}
\end{proof}

\section{Recovering the \'etale Euler characteristic with compact support}\label{sec:etale with compact support}
An essential feature of the categorical Euler characteristic $\chi^\cat:\SH^\rig(S)\to\End(\1_S)$ is its connection to the \'etale Euler characteristic. When $S$ is a subfield of $\mb{C}$, we have $\rank\chi^\cat(\Sigma^\infty_{\mb{P}^1}X_+)=\chi(X(\mb{C})^{\mr{an}})$. If $S$ is a subfield of $\mb{R}$, we also have $\sign\chi^\cat(\Sigma^\infty_{\mb{P}^1}X_+)=\chi(X(\mb{R})^{\mr{an}})$. That is, the rank and signature of the categorical Euler characteristic of a smooth scheme are given by the singular Euler characteristic of its complex and real loci, respectively \cite[Remarks 2.3 (1)]{Levine-quadratic}. 

Over any field of characteristic not 2 the rank of $\chi^\cat$ is equal to the \'etale Euler characteristic  \cite[Section 1]{Levine-quadratic}. In this subsection, we will expand on Levine's remarks in \textit{loc.~cit.}~to show that the rank of $\chi^c$ is equal to the \'etale Euler characteristic with compact support. This will imply that the image of $\chi^c$ is contained in $\GW(k)\subseteq\GW(k)[e^{-1}]$ (see Corollary~\ref{cor:gw inside gw[1/p]}).

\begin{prop}\label{prop:rank = etale}
Let $k$ be a field of characteristic not 2. Let $\ell$ be a prime that not equal to $\Char{k}$. If $X$ is a $k$-variety, then $\rank\chi^c([X])=\chi^\et(X)$, where $\chi^\et$ is the \'etale Euler characteristic with compact support:
\begin{equation}\label{eq:alternating sum}
\chi^\et(X):=\sum_i(-1)^i\dim \H^i_c(X_{k^\mr{sep}};\mb{Q}_\ell).
\end{equation}
In particular, we obtain a ring homomorphism $\rank\chi^c:K_0(\Var_k)\to\mb{Z}$.
\end{prop}
\begin{proof}
The proof is essentially identical to Levine's argument for $\rank\chi^\cat$ given in \cite[Section 1]{Levine-quadratic}. We will expand on the details for the reader's convenience. Let $q:X\to\Spec{k}$ be a quasi-projective $k$-variety. Let $\mr{D}^b_c(-,\mb{Q}_\ell)$ denote the derived category of constructible sheaves with $\mb{Q}_\ell$-coefficients. For any prime $\ell$ not equal to $e$, the $\ell$-adic realization functor $r_\ell:\SH(k)[e^{-1}]\to\mr{DM}(k)[e^{-1}]\to\mr{D}^b_c(k,\mb{Q}_\ell)$ sends the motivic spectrum $q_!\1_X\in\SH(k)[e^{-1}]$ to the complex $q_!\mb{Q}_{\ell,X}\in\mr{D}^b_c(k,\mb{Q}_\ell)$. The key point is that the complex $q_!\mb{Q}_{\ell,X}=q_!q^*\mb{Q}_{\ell}$ represents $\ell$-adic cohomology with compact supports (see e.g.~\cite[Th\'eor\`eme (5.4)]{Deligne}, or \cite[\S 1.4]{Gallauer} for a modern survey).

Now we apply the previous discussion to the base change $\bar{q} : X_{k^\mr{sep}} \to \Spec{k^{\mr{sep}}}$ to the separable closure to see that $\bar{q}_{!}\mb{Q}_{\ell, X_{k^\mr{sep}}} \in \mr{D}^b_c(k^\mr{sep},\mb{Q}_\ell)$ represents $\H^\bullet_c(X_{k^\mr{sep}};\mb{Q}_\ell)$ and so the categorical trace $\tr(\id_{\bar{q}_!\mb{Q}_{\ell,X}})$ is the alternating sum $\chi^{\et}(X)$ in Equation~\ref{eq:alternating sum}. On the other hand, we can identify the map $\rank$ with the base extension map $i_*:\GW(k) \to \GW(k^\mr{sep}) \cong \mathbb{Z}$, as $k^\mr{sep}$ is quadratically closed when $\Char{k}\neq 2$. By Proposition~\ref{prop:field_ext}, we have $\rank \chi^c(X) = \chi^c(X_{k^\mr{sep}})$ and by Proposition~\ref{prop:chi on quasi proj}, $\chi^c(X_k^\mr{sep}) = \tr(\id_{\bar{q}_!\1_{X_{k^\mr{sep}}}})$. 

Now note that the $\ell$-adic realization functor $r_{\ell}$ induces a canonical map $\GW(k^\mr{sep})[e^{-1}] \to \End_{\mr{D}^b_c(k^\mr{sep},\mb{Q}_\ell)}(\mb{Q}_\ell) = \mb{Q}_{\ell}$ which may be identified with the inclusion $\mb{Z}[e^{-1}] \to \mb{Q}_{\ell}$. Moreover, categorical traces are compatible with symmetric monoidal functors, so under this inclusion $\tr(\id_{\bar{q}_!\1_{X_{k^\mr{sep}}}})$ gets identified with $\tr(\id_{\bar{q}_!\mb{Q}_{\ell,X}})$. Thus $\rank \chi^c(X) = \chi^{\et}(X)$ as required. 

Next, we note that both $\rank\chi^c$ and $\chi^\et$ induce ring homomorphisms
\[\rank\chi^c,\chi^\et:K_0(\Var_k)\to\mb{Z}[e^{-1}].\]
We have proved that these two ring homomorphisms agree on quasi-projective varieties, so $\rank\chi^c=\chi^\et$ by \cite[Chapter 2, Corollary 2.6.6]{CLNS:MotivicIntegration}. 

The compactly supported $\ell$-adic Euler characteristic of any $k$-variety is a sum of integers, so $\rank\chi^c([X])\in\mb{Z}$ for any $k$-variety $X$. Since $K_0(\Var_k)$ is generated as a ring by classes of $k$-varieties, the image of $\rank\chi^c$ is generated by its value on classes of $k$-varieties. In particular, $\rank\chi^c(M)\in\mb{Z}$ for any $M\in K_0(\Var_k)$. 
\end{proof}

\begin{rem}\label{rem:rank}
If $k\subseteq\mb{C}$, the Betti realization functor
\begin{align*}
    \mr{re}_\mb{C}:\SH(\mb{C})&\to\SH\\
    \Sigma^\infty_{\mb{P}^1}X_+&\mapsto\Sigma^\infty X(\mb{C})_+
\end{align*}
sends the map of varieties $q:X\to\Spec{k}$ to the map $q(\mb{C}):X(\mb{C})\to\{\mr{pt}\}$ of topological spaces. Singular cohomology with compact supports of $X$ is represented by the functor $q(\mb{C})_!$, so we find that $\mr{re}_\mb{C}$ sends $\chi^c(X_{\mb{C}})$ to the compactly supported Euler characteristic of the topological space $X(\mb{C})$. Together with the observation that
\[\mr{re}_\mb{C}:\GW(\mb{C})\cong\End_{\SH(\mb{C})}(\1_\mb{C})\to\End_\SH(\mb{S})\cong\mb{Z}\]
is the rank homomorphism (cf. \cite[Remarks 2.3 (1)]{Levine-quadratic}), we get a slightly different version of Proposition~\ref{prop:rank = etale} in this context.
\end{rem}

Because $\rank\chi^c$ is $\mb{Z}$-valued, we can deduce that $\chi^c$ is $\GW(k)$-valued when $\Char{k}\neq 2$.

\begin{cor}\label{cor:gw inside gw[1/p]}
If $k$ is a field with odd exponential characteristic $e$, then the image of $\chi^c:K_0(\Var_k)\to\GW(k)[e^{-1}]$ is contained in $\GW(k)\subseteq\GW(k)[e^{-1}]$.
\end{cor}
\begin{proof}
This follows from the discussion in \cite[Remarks 2.1 (2)]{Levine-quadratic}, which we briefly explain here. When $\Char{k}\neq 2$, the map $\GW(k)\to\GW(k)[e^{-1}]$ is injective, fitting into the bi-cartesian commutative diagram
\[\begin{tikzcd}
    \GW(k)\arrow[r,hook]\arrow[d,"\rank"'] & \GW(k)[e^{-1}]\arrow[d,"\rank"]\\
    \mb{Z}\arrow[r,hook] & \mb{Z}[e^{-1}].
\end{tikzcd}\]
Since the image of $\rank\chi^c$ is contained in $\mb{Z}\subseteq\mb{Z}[e^{-1}]$, this implies that the image of $\chi^c$ is contained in $\GW(k)\subseteq\GW(k)[e^{-1}]$.
\end{proof}

As one might hope, the signature of $\chi^c$ recovers the compactly supported Euler characteristic of the real locus (or more generally, the real closed locus).

\begin{prop}\label{prop:sign = etale of real locus}
Let $k$ be a field admitting an embedding $\sigma:k\to R$, where $R$ is a real closed field. If $X$ is a $k$-variety, then $\sign_\sigma\chi^c([X])=\chi^t(X_\sigma(R))$, where $\chi^t$ is the Euler characteristic with compact support for topological spaces, $X_\sigma(R)$ is the real closed locus of $X$ under the embedding $\sigma$, and $\sign_\sigma$ is the signature with respect to the embedding $\sigma$.
\end{prop}
\begin{proof}
The real closed embedding $\sigma:k\to R$ induces a base change functor\footnote{Here, we conflate notation and also write $\sigma:\Spec{R}\to\Spec{k}$.} $\sigma^*:\SH(k)\to\SH(R)$, which is symmetric monoidal. There is a real closed Betti realization functor
\begin{align*}
    \mr{re}_R:\SH(R)&\to\SH\\
    \Sigma^\infty_{\mb{P}^1}X_+&\mapsto\Sigma^\infty X_\sigma(R)_+,
\end{align*}
where $X_\sigma(R)$ is the topological space underlying the semialgebraic set $X(R)$ under the ordering on $R$. When $R=\mb{R}$, this Betti realization functor is real Betti realization, which is well-known. Using the same argument outlined in Remark~\ref{rem:rank}, the composite $\mr{re}_R\circ\sigma^*$ sends $\chi^c(X)$ to $\chi^t(X_\sigma(R))$. Because $\sigma^*$ is monoidal, the real closed embedding induces $\sigma^*:\End_{\SH(k)}(\1_k)\to\End_{\SH(R)}(\1_R)$, which is the homomorphism $\GW(k)\to\GW(R)$ used to compute $\sign_\sigma$. We will thus have the desired result as soon as we know that 
\[\mr{re}_R:\GW(R)\cong\End_{\SH(R)}(\1_R)\to\End_\SH(\mb{S})\cong\mb{Z}\]
is the signature homomorphism. For $R=\mb{R}$, this is an observation of Morel, recorded in \cite[Remarks 2.3 (1)]{Levine-quadratic}: the automorphism $[x:y]\mapsto[x:\pm y]$ of $\mb{P}^1_k$ corresponds to $\langle\pm 1\rangle\in\GW(k)$ under Morel's degree map, and the real realization of this automorphism is the map
\begin{align}\label{eq:involution}
\begin{split}
    S^1&\mapsto S^1\\
    \theta&\mapsto \pm\theta,
\end{split}
\end{align}
which has Brouwer degree $\pm 1$. Thus $\mr{re}_\mb{R}$ induces the map
\begin{align*}
    \GW(\mb{R})&\to\mb{Z}\\
    \langle\pm 1\rangle&\mapsto\pm 1,
\end{align*}
which is the signature homomorphism. The same argument holds verbatim for arbitrary real closed fields, as the automorphism $[x:y]\to[x:\pm y]$ of $\mb{P}^1_k$ is defined over $\mb{Z}$, and any real closed realization of this map is again Equation~\ref{eq:involution}.
\end{proof}

\begin{rem}
A key observation of Pajwani--P\'al is that in characteristic 0, the discriminant of $\chi^c$ recovers Saito's determinant of $\ell$-adic cohomology, appropriately scaled \cite[Theorem 2.27]{PajwaniPal:YauZaslow}. This follows from work of Saito (see e.g.~\cite{Sai94,Sai97,Sai12,Ter18}). 

The necessary results of Saito are known for smooth proper varieties, whose classes generate $K_0(\Var_k)$ in characteristic 0. In positive characteristic, one would need analogous results for all quasi-projective varieties (which may be assumed to be smooth if one works over a perfect field and inverts $e$.). Alternatively, one could search for a realization functor sending the trace $\tr(\id_{q_!\1_X})$ to the determinant of the complex represented by $q_!\1_X$. Potentially relevant articles in this direction include \cite{Bei07,Bre11,MTW15}.
\end{rem}

\begin{rem}
There are many other invariants of quadratic forms beyond rank, signature, and discriminant. Taking any such invariant of $\chi^c(X)$ gives an invariant of $X$, and one can ask whether such an invariant already occurs in the literature. For example, the Hasse--Witt invariants of $\chi^c(X)$ have been studied by Saito \cite{Sai13}.
\end{rem}

\section{Modifications of $K_0(\Var_k)$}\label{sec:modifications}
There are various ways that $K_0(\Var_k)$ can be modified. For example, one can invert universal homeomorphisms, localize at $[\mb{A}^1_k]$, complete with respect to the dimension filtration, or localize at a topology finer than the Zariski topology. In this section, we will discuss a few modifications of $K_0(\Var_k)$ through which $\chi^c$ factors.

\subsection{Inverting universal homeomorphisms}\label{sec:universal homeo}
In positive characteristic, the lack of resolution of singularities poses a problem for many arguments that we might hope to make. A common workaround is to invert those maps that might prevent an alteration, which exist over any field, from being a resolution of singularities. To this end, we will recall a certain modification of the Grothendieck ring of varieties and show that $\chi^c$ factors through this modification.

\begin{defn}
A morphism $f:X\to Y$ is called a \textit{universal homeomorphism} if $f_{Y'}:X\times_YY'\to Y'$ is a homeomorphism (of underlying topological spaces) for every scheme $Y'\to Y$. Equivalently, $f$ is integral, surjective, and \textit{universally injective}, i.e.~$f|_{X(K)}:X(K)\to Y(K)$ is injective for every field $K$.
\end{defn}

\begin{defn}
Let $S$ be a scheme. Let
\[I^\uh_S:=([X]-[Y]:\exists\text{ universal homeomorphism }f:X\to Y)\]
be the ideal in $K_0(\Var_S)$ generated by differences of classes of universally homeomorphic $S$-varieties. Define the \textit{Grothendieck ring of $S$-varieties up to universal homeomorphism} as the quotient $K^\uh_0(\Var_S):=K_0(\Var_S)/I^\uh_S$.
\end{defn}

See \cite[Chapter 2, \S 4]{CLNS:MotivicIntegration} for more details on $K^\uh_0(\Var_S)$. In characteristic 0, the ideal $I^\uh_S$ is trivial, so that $K^\uh_0(\Var_S)\cong K_0(\Var_S)$ \cite[Chapter 2, Corollary 4.4.7]{CLNS:MotivicIntegration}. It is unknown whether $I^\uh_k=(0)$ when $\Char{k}>0$, but the following lemma and its corollary will be useful for us.

\begin{lem}\label{lem:r_mot universal homeo}
Let $S$ be a noetherian scheme. Let $\mc{P}$ denote the set of primes not invertible in $\mc{O}_S$. Let $L_\mc{P}:K_0^\triangle(\SH^\omega(S))\to K_0^\triangle(\SH^\omega(S)[\mc{P}^{-1}])$ denote localization at the set of non-invertible primes. Then $L_\mc{P}\circ r_\mot:K_0(\Var_S)\to K_0^\triangle(\SH^\omega(S)[\mc{P}^{-1}])$ factors through $K^\uh_0(\Var_S)$.
\end{lem}
\begin{proof}
By Lemma~\ref{lem:motivic realization}, it suffices to prove that if $p:X\to S$ and $q:Y\to S$ are quasi-projective varieties that are universally homeomorphic, then $[p_!\1_X]=[q_!\1_Y]$ in $K_0^\triangle(\SH^\omega(S)[\mc{P}^{-1}])$. Let $f:X\to Y$ be a universal homeomorphism, so that we have a commutative triangle
\[\begin{tikzcd}
X\arrow[rr,"f"]\arrow[dr,"p"'] && Y.\arrow[dl,"q"]\\
& S &
\end{tikzcd}\]
Any prime not invertible in $\mc{O}_X$ must be contained in $\mc{P}$, so \cite[Proof of Theorem 2.1.1]{ElmantoKhan} implies that the unit $\mr{id}\to f_*f^*$ of the adjunction
\[f^*:\SH(Y)[\mc{P}^{-1}]\rightleftarrows\SH(X)[\mc{P}^{-1}]:f_*\]
is invertible. In particular, $q_!\1_Y\cong q_!f_*f^*\1_Y$. Since $f^*$ is monoidal, this implies that $q_!\1_Y\cong q_!f_*\1_X$. If $f$ were proper, then $f_*\simeq f_!$. We would then have $q_!f_*\1_X\cong q_!f_!\1_X\cong p_!\1_X$ and hence $q_!\1_Y\cong p_!\1_X$. This would immediately imply that $[p_!\1_X]=[q_!\1_Y]$ in $K_0^\triangle(\SH^\omega(S)[\mc{P}^{-1}])$.

It thus remains to show that a universal homeomorphism between quasi-projective schemes is proper. By \cite[\href{https://stacks.math.columbia.edu/tag/01VX}{Lemma 01VX}]{stacks} and \cite[\href{https://stacks.math.columbia.edu/tag/01T8}{Lemma 01T8}]{stacks}, any morphism between quasi-projective schemes is locally of finite type. By \cite[\href{https://stacks.math.columbia.edu/tag/04DE}{Lemma 04DE}]{stacks} a universal homeomorphism is affine and in particular quasi-compact and separated. Finally, it is universally closed by definition so it is proper.
\end{proof}

\begin{cor}\label{cor:chi universal homeo}
Let $k$ be a field. Then $\chi^c$ factors through $K^\uh_0(\Var_k)$.
\end{cor}
\begin{proof}
If $\Char{k}=0$, then $K_0(\Var_k)\cong K^\uh_0(\Var_k)$ and there is nothing to show. Otherwise $\mc{P}=\{\Char{k}\}$. Since $\chi^c=\chi^\omega\circ L_e\circ r_\mot$, the result follows from Lemma~\ref{lem:r_mot universal homeo}.
\end{proof}

\begin{rem}
Corollary~\ref{cor:chi universal homeo} implies that various results in \cite{PRV24} extend from characteristic 0 to all odd characteristics.
\end{rem}

\subsection{Localization at the Lefschetz motive}
We now show that $\chi^c$ is insensitive to localization at $[\mb{A}^1_k]$ \cite[Chapter 2, \S 4.2]{CLNS:MotivicIntegration}.

\begin{defn}
Let $S$ be a scheme. Let $\mb{L}_S:=[\mb{A}^1_S]\in K_0(\Var_S)$, which is called the \textit{Lefschetz motive}. Let $\mathscr{M}_S:=K_0(\Var_S)[\mb{L}_S^{-1}]$.
\end{defn}

\begin{prop}\label{prop:chi lefschetz local}
Let $k$ be a field. Then $\chi^c$ factors through $\mathscr{M}_k$.
\end{prop}
\begin{proof}
By the universal property of $\mathscr{M}_k$, it suffices to verify that $\chi^c(\mb{L}_k)$ is invertible in $\GW(k)[e^{-1}]$ \cite[Paragraph (4.2.1)]{CLNS:MotivicIntegration}. We have computed $\chi^c(\mb{L}_k)=\langle -1\rangle$ (Corollary~\ref{cor:chi A^n}), which has multiplicative inverse $\langle -1\rangle$.
\end{proof}

Note that we can similarly localize the Grothendieck ring up to universal homeomorphisms as well.

\begin{defn}
Let $S$ be a scheme. Let $\mb{L}^\uh_S:=[\mb{A}^1_S]\in K^\uh_0(\Var_S)$. We can then define $\mathscr{M}^\uh_S:=K^\uh_0(\Var_S)[(\mb{L}^\uh_S)^{-1}]$.
\end{defn}

\begin{prop}\label{prop:chi uh lefschetz local}
Let $k$ be a field. Then $\chi^c$ factors through $\mathscr{M}^\uh_k$.
\end{prop}
\begin{proof}
We have already seen that $\chi^c$ factors through $K^\uh_0(\Var_k)$ (Corollary~\ref{cor:chi universal homeo}) and through $\mathscr{M}_k$ (Proposition~\ref{prop:chi lefschetz local}). It thus suffices to note that the following diagram
\begin{equation}\label{eq:muh pushout}
\begin{tikzcd}
K_0(\Var_k)\arrow[r,"-/I^\uh_k"]\arrow[d,"\mb{L}^{-1}_k"'] & K^\uh_0(\Var_k)\arrow[d,"(\mb{L}^\uh_k)^{-1}"]\\
\mathscr{M}_k\arrow[r,"-/I^\uh_k"] & \mathscr{M}^\uh_k
\end{tikzcd}
\end{equation}
is a pushout, since this is a diagram of commutative rings and localization commutes with quotients for commutative rings. Thus there exists a unique dashed arrow completing the following diagram:
\begin{equation}\label{eq:muh factors}
\begin{tikzcd}
K_0(\Var_k)\arrow[r]\arrow[d] & K^\uh_0(\Var_k)\arrow[d]\arrow[ddr,bend left=20,"\text{Corollary~\ref{cor:chi universal homeo}}"] &\\
\mathscr{M}_k\arrow[r]\arrow[drr,bend right=20,"\text{Proposition~\ref{prop:chi lefschetz local}}"'] & \mathscr{M}^\uh_k\arrow[dr,dashed] &\\
&& \GW(k)[e^{-1}].
\end{tikzcd}
\end{equation}
Diagram~\ref{eq:muh factors} commutes, and hence $\chi^c$ factors through $\mathscr{M}^\uh_k$.
\end{proof}

\subsection{Topological localization}
As mentioned in the proof of Corollary~\ref{cor:chi mult fiber bundles}, classes in $K_0(\Var_k)$ factor over Zariski-locally trivial fiber bundles. Given a topology $\tau$ finer than the Zariski topology, one can impose multiplicativity over $\tau$-locally trivial fiber bundles to obtain the quotient ring $K_0^\tau(\Var_k)$. One can then ask whether a motivic measure factors through $K_0^\tau(\Var_k)$.

\begin{defn}
Let $\tau$ be a Grothendieck topology on $\Var_k$. A finite type morphism $f:Y\to X$ is called a \textit{$\tau$-locally trivial fiber bundle with fiber $F$} if there exist a $\tau$ covering $\{p_i:U_i\to X\}_{i\in I}$ and isomorphisms of $U_i$-schemes $Y\times_X U_i\cong F\times_k U_i$ for all $i\in I$.

Now consider the ideal $I^\tau\subset K_0(\Var_k)$ generated by classes of the form $[E]-[F]\cdot [B]$, where $E\to B$ is a $\tau$-locally trivial fiber bundle with fiber $F$. Define the \textit{$\tau$ local Grothendieck ring of varieties} as the quotient 
\[K_0^\tau(\Var_k):=K_0(\Var_k)/I^\tau.\]
A motivic measure that factors through $K_0^\tau(\Var_k)$ is said to be \textit{$\tau$ local}.
\end{defn}

The \'etale topology drastically culls classes in $K_0(\Var_k)$. Indeed, the \'etale Euler characteristic induces an isomorphism $K_0^\et(\Var_k)\to\mb{Z}$ in characteristic 0 (see e.g.~\cite{EvgenyShinder}). In particular, if a motivic measure $\mu$ is strictly more interesting than the \'etale Euler characteristic, then $\mu$ should not be \'etale local. The motivic Euler characteristic is an example of such a motivic measure.

\begin{prop}
Let $k$ be a field of characteristic 0. Assume that $k$ is not quadratically closed. Then $\chi^c$ does not factor through $K_0^\et(\Var_k)$. 
\end{prop}
\begin{proof}
We will later prove that $\chi^c$ surjects onto $\GW(k)$ (Proposition~\ref{prop:chi^c surjective}), which surjects onto but is not isomorphic to $\mb{Z}$ whenever $k$ is not quadratically closed. In particular, $\chi^c$ cannot factor through $K_0^\et(\Var_k)\cong\mb{Z}$.
\end{proof}

Sitting between the Zariski and \'etale topologies is the Nisnevich topology, which is built from \'etale covers that induce isomorphisms on residue fields of points. Recall that the unstable motivic homotopy category consists of $\mb{A}^1$-local Nisnevich sheaves on smooth schemes, and that $\chi^c$ factors through the Grothendieck ring of the stable motivic homotopy category. It is therefore unsurprising that $\chi^c$ is Nisnevich local. However, it turns out that all motivic measures are Nisnevich local.

\begin{prop}\label{prop:Nisnevich local}
Let $\Nis$ denote the Nisnevich topology. There is a ring isomorphism $K_0(\Var_k)\cong K_0^\Nis(\Var_k)$, so any motivic measure is Nisnevich local.
\end{prop}
\begin{proof}
Let $f:E\to B$ be a Nisnevich-locally trivial fibration with fiber $F$. We will show that $[E]=[F]\cdot[B]$ in $K_0(\Var_k)$, which will yield the claimed ring isomorphism. Finite type schemes over a field are noetherian,\footnote{By convention, $\Var_S$ is the category of $S$-varieties of finite presentation, which implies finite type.} so $B$ has finitely many irreducible components $B_1,\ldots,B_n$. When a scheme has finitely many irreducible components, these components are pairwise disjoint and hence $[B]=[B_1]+\ldots+[B_n]$. Since $\chi^c$ is additive, we may therefore assume that $B$ is irreducible. Let $\eta_B$ denote the generic point of $B$.

Let $\mc{U}:=\{p_i:U_i\to B\}_{i\in I}$ be a Nisnevich cover that trivializes $f$. Nisnevich covers are jointly surjective on all (not just closed) points, so there exists $(p,U)\in\mc{U}$ with $p^{-1}(\eta_B)\in U$. In particular, there exists an irreducible component $V$ of $U$ with generic point $\eta_V$ such that $p$ induces an isomorphism $k(\eta_V)\cong k(\eta_B)$. In other words, $p|_V:V\to B$ is a birational morphism.

By further restricting $p|_V$ to the open locus of definition, there exist non-empty open subschemes $V'\subseteq V$ and $B'\subseteq B$ such that $p|_{V'}:V'\to B'$ is an isomorphism. Since $f$ is trivial over $U$ by assumption and $V'\subset U$ is an open immersion, we have $E\times_B V'\cong F\times_k V'$. Since $V'\cong B'$ as $k$-schemes, we thus have $E\times_B B'\cong F\times_k B'$. Thus $[f^{-1}(B')]=[F]\cdot[B']$. Since $B'$ is a non-empty open subset of $B$, the complement $B_1 := B\setminus B'$ is a proper closed subscheme of $B$. Then $[E] = [f^{-1}(B')] + [f^{-1}(B_1)]$ and the restriction of $\mc{U}$ to $B_1$ gives a Nisnevich cover which trivializes $f$. Thus we may repeat the argument on $B_1$ to obtain a non-empty open $B_1' \subset B_1$ with $[f^{-1}(B_1')] = [F]\cdot[B_1']$ and set $B_2 = B_1 \setminus B_1'$ and repeat. Continuing in this way, we get a decreasing sequence of closed subschemes $\ldots \subset B_{m+1} \subset B_m \subset \ldots \subset B_0 = B$ with complements $B_m' = B_m \setminus B_{m+1}$ (where we set $B' =\colon B_0'$) such that $[f^{-1}(B_m')] = [F] \cdot [B_m']$. Since $B$ is noetherian, this chain must terminate with $B_N' = B_N$ for some $N \gg 0$ and so 
\[[E] = \sum_{m = 0}^N [f^{-1}(B_m')] = [F] \cdot \sum_{m=0}^N [B_m'] = [F] \cdot [B].\qedhere\]
\end{proof}

Finer than the Nisnevich topology (but incomparable to the \'etale topology) is the cdh topology. This is obtained from the Nisnevich topology by adding in covers $\{p:Y\to X,i:Z\to X\}$, where $p$ is proper, $i$ is a closed immersion,
\begin{equation}\label{eq:blow up square}
\begin{tikzcd}
    E\arrow[r]\arrow[d] & Y\arrow[d,"p"]\\
    Z\arrow[r,"i"] & X
\end{tikzcd}
\end{equation}
is a Cartesian diagram, and $Y-E\cong X-Z$. Diagram~\ref{eq:blow up square} is called an \textit{abstract blow up square}. Since $[E]=[Y]+[Z]-[X]\in K_0(\Var_k)$ for any abstract blow up square, it follows that $K_0(\Var_k)$ cannot distinguish between the Nisnevich and cdh topologies.

\begin{prop}\label{prop:cdh}
There is a ring isomorphism $K_0^\Nis(\Var_k)\cong K_0^\cdh(\Var_k)$, so any motivic measure is $\cdh$ local.
\end{prop}
\begin{proof}
By \cite[Proposition 7.10]{Kui23}, the ring $K_0(\Var_k)$ can be obtained using proper varieties as generators and abstract blow up squares as relations. Since every cover in the cdh topology is either an abstract blow up square or a Nisnevich cover, we find that $K_0^\Nis(\Var_k)\cong K_0^\cdh(\Var_k)$. The fact that any motivic measure is cdh local now follows from Proposition~\ref{prop:Nisnevich local}.
\end{proof}

\begin{rem}
After a suitable stratification, $X^n\to\Sym^n X$ is \'etale-locally trivial, but this only tells us that $\rank\chi^c(\Sym^n X)=0$ if $\rank \chi^c(X) = 0$. In order to attack Conjecture~\ref{conj:main}, one might hope to find a coarser topology in which $X^n\to\Sym^n X$ is locally trivial (after a suitable stratification). Since $\chi^c$ factors through $K_0^\Nis(\Var_k)$, the Nisnevich topology is the first place to look. 

Unfortunately, $X^n\to\Sym^n X$ (for $n\geq 2$) is not Nisnevich-locally trivial after any stratification. Indeed, over the open stratum of $\Sym^nX$ parametrizing reduced subschemes of length $n$, the map $X^n\to\Sym^nX$ restricts to a degree $n!$ \'etale cover with irreducible source. This yields non-trivial extensions of residue fields, which prevents the \'etale-locally trivial fiber bundle from being Nisnevich-locally trivial. This same issue arises with any completely decomposed topology, such as the cdh topology.
\end{rem}

\section{A power structure on $K_0(\Var_k)$}\label{sec:power structure K0}
A key input for Beauville's proof of the Yau--Zaslow formula \cite{Beauville:Yau-Zaslow} is G\"ottsche's formula, which implies that for a smooth complex surface $X$, the generating series of the Euler characteristics of the Hilbert schemes $\Hilb^nX$ is given by $\prod_{i\geq 1}(1-t^i)^{-\chi(X)}$. The power structure on $K_0(\Var_\mb{C})$ \cite{PowerStructure-K0(Var)} enables one to give a very natural proof of G\"ottsche's formula. This power structure allows one to make sense of (and prove \cite{GZLMH:hilbert}) the expression
\[\sum_{g\geq 0}[\Hilb^gX]\cdot t^g=(\sum_{m\geq 0}[\Hilb^m_0(\mb{A}^d)]\cdot t^m)^{[X]}\]
in $K_0(\Var_\mb{C})$, where $d$ is any smooth quasi-projective variety of dimension $d$. Furthermore, the Euler characteristic with compact supports is compatible with this exponentiation, so that
\[\sum_{g\geq 0}\chi(\Hilb^gX)\cdot t^g=(\sum_{m\geq 0}\chi(\Hilb^m_0(\mb{A}^d))\cdot t^m)^{\chi(X)}.\]

The key to proving this equality is that the Euler characteristic with compact support induces a surjective ring homomorphism $\chi:K_0(\Var_\mb{C})\to\mb{Z}$ that respects the power structure on the source and target. Our ultimate goal for the next few sections is to show that over any field $k$ with $\Char{k}\neq 2$, there is a power structure on $\GW(k)$ (assuming Conjecture~\ref{conj:main}) that is compatible with a power structure on $K_0(\Var_k)$ and the motivic Euler characteristic $\chi^c$. This will enable us to prove, conditional on Conjecture~\ref{conj:main}, an enriched G\"ottsche formula en route to the arithmetic Yau--Zaslow formula. 

To begin, we will provide some exposition around the details of \cite{PowerStructure-K0(Var)}. Rather than restricting our attention to $\mb{C}$, we will work over an arbitrary field, addressing any positive characteristic anomalies that might arise. Many of these details in arbitrary characteristic are present in \cite{GZLMH:hilbert,GZLMH13} under the added assumption that $k$ is algebraically closed, but we include and expand on them here for the reader's convenience.

\begin{defn}\label{def:power structure}
Let $R$ be a ring. Let $1+t\cdot R\dlb t\drb$ denote the (multiplicatively closed) set of formal power series of the form $1+\sum_{i=1}^\infty a_it^i$ with $a_i\in R$. A \textit{power structure} on $R$ is a function
\[(1+t\cdot R\dlb t\drb)\times R\to 1+t\cdot R\dlb t\drb,\]
denoted by $(A(t),r)\mapsto A(t)^r$, satisfying the following properties for all $A(t),B(t)\in 1+t\cdot R\dlb t\drb$ and $r,s\in R$:
\begin{enumerate}[(i)]
\item $A(t)^0=1$.
\item $A(t)^1=A(t)$.
\item $(A(t)\cdot B(t))^r=A(t)^r\cdot B(t)^r$.
\item $A(t)^{r+s}=A(t)^r\cdot A(t)^s$.
\item $A(t)^{rs}=(A(t)^r)^s$.
\item $(1+t)^r=1+rt+o(t^2)$.
\item $A(t^i)^r=A(t)^r|_{t\mapsto t^i}$ for all $i\geq 1$.
\end{enumerate}
Further, a power structure is called \emph{finitely determined} if for each $i>0$, there exists $j>0$ such that for all $A(t)\in 1+t\cdot R\dlb t\drb$ and $m\in R$, the series $A(t)^m\mod t^{i+1}$ is determined by $A(t)\mod t^{j+1}$.
\end{defn}

\begin{rem}
All of the power structures that we will encounter in this paper are finitely determined, so we will generally not mention this assumption explicitly.    
\end{rem}

\begin{ex}
When $R=\mb{Z}$, the usual exponentiation of power series determines a power structure on $\mb{Z}$. In fact, Gusein-Zade--Luengo--Melle-Hern\'andez use this power structure to guide their construction of the power structure on $K_0(\Var_\mb{C})$ \cite[p.~50]{PowerStructure-K0(Var)}. Let $r\in\mb{Z}$ and $1+\sum_{i=1}^\infty a_it^i\in 1+t\cdot\mb{Z}\dlb t\drb$. Inducting on the multinomial theorem, we find that for $r\geq 0$,
\begin{align}\label{eq:power structure on Z}
\left(1+\sum_{i=1}^\infty a_it^i\right)^r&=1+\sum_{n=1}^\infty\left(\sum_{\substack{b_1,\ldots,b_n\geq 0;\\ \sum_i ib_i=n}}\binom{r}{b_1,\ldots,b_n,r-\sum_ib_i}\prod_{i=1}^n a_i^{b_i}\right)t^n,
\end{align}
where each $\binom{r}{b_1,\ldots,b_n,r-\sum_ib_i}=\frac{r(r-1)\cdots(r-\sum_ib_i+1)}{b_1!\cdots b_n!}$ is a multinomial coefficient. Note that for the degree $n$ term, we only need to consider the multinomial theorem for $(1+a_1t+\ldots+a_nt^n)^r$, since all remaining terms will have degree strictly greater than $n$. For $-r<0$, we set $A(t)^{-r}:=\frac{1}{A(t)^r}$. One can check by induction on polynomial division that $A(t)^{-r}\in 1+t\cdot\mb{Z}\dlb t\drb$.
\end{ex}

We now want to adapt Equation~\ref{eq:power structure on Z} to $R=K_0(\Var_k)$, where $k$ is an arbitrary field. Since $a_i,r\in R$, these should be replaced with isomorphism classes of quasi-projective varieties $[A_i],[X]\in K_0(\Var_k)$. If we continue to view $b_i$ as integers, a natural replacement of $a_i^{b_i}$ is the product variety $A_i^{b_i}$. The real work comes from correctly interpreting the multinomial coefficient $\binom{X}{b_1,\ldots,b_n,X-\sum_ib_i}$. Over an algebraically closed field, this should be a variety whose Euler characteristic is $\binom{\chi(X)}{b_1,\ldots,b_n,\chi(X)-\sum_ib_i}$. Such a variety is given by $(\prod_{i=1}^n X^{b_i}-\Delta)/\prod_{i=1}^n S_{b_i}$, where the \textit{big diagonal} $\Delta$ is the union of the pairwise diagonals in $\prod_i X^{b_i}$. For an example of this computation in a concrete case, see e.g.~\cite{dC00}.

Based on these observations, one is led to propose the following power structure on isomorphism classes of quasi-projective varieties (i.e. the Grothendieck semiring of \textit{effective} classes).

\begin{defn}\label{def:power structure on effective classes}
    Let $k$ be an arbitrary field. Let $X$ and $A_1,A_2,\ldots$ be quasi-projective varieties over $k$. Let $A(t):=1+\sum_i[A_i]t^i\in 1+t\cdot K_0(\Var_k)\dlb t\drb$. Define
    \[\mu_\mr{eff}(A(t),[X]):=1+\sum_{n=1}^\infty\left(\sum_{\substack{b_1,\ldots,b_n\geq 0;\\ \sum_i ib_i=n}}\left[\frac{((\prod_{i=1}^n X^{b_i})-\Delta)\times\prod_{i=1}^n A_i^{b_i}}{\prod_{i=1}^n S_{b_i}}\right]\right)t^n,\]
    where $\Delta$ is the union of the pairwise diagonals in $\prod_i X^{b_i}$ and $S_{b_i}$ acts by permuting the $b_i$ factors in $\prod_i X^{b_i}-\Delta$ and $A_i^{b_i}$.
\end{defn}

\begin{rem}
    Over an algebraically closed field, one can describe $\Delta$ as the set of points in $\prod_i X^{b_i}$ with at least two coinciding coordinates. Over an arbitrary field, we no longer have a clean description of $\Delta$ as the boundary of a configuration space, so we instead construct $\Delta$ in terms of the various diagonal morphisms.
\end{rem}

Note that $\mu_0$ satisfies the criteria of Definition~\ref{def:power structure}, at least when the exponents involved are classes of quasi-projective varieties. In order to turn $\mu_\mr{eff}$ into a genuine power structure, it thus suffices to define $\mu_\mr{eff}$ when the exponent is a general element of $K_0(\Var_k)$.

Rather than explicitly extending the definition of $\mu_\mr{eff}$, we will define $(1-t)^{-[X]}$ for every quasi-projective variety $X$. By \cite[Proposition 1]{GZLMH:hilbert}, defining the power series $(1-t)^{-[X]}$ for all quasi-projective varieties $X$ uniquely determines a (finitely determined) power structure on $K_0(\Var_k)$. We start with a definition.

\begin{defn}\label{def:sym_detla}
Let $n\geq 1$, and let $X$ be a quasi-projective variety over a field $k$. Define
\[\Sym_\Delta^n X:=\coprod_{\substack{a_1,\ldots,a_n\geq 0\\\sum_i ia_i=n}}\left(\left(\prod_{i=1}^n X^{a_i}-\Delta\right)/\prod_{i=1}^n S_{a_i}\right),\]
where $\Delta$ is the union of the pairwise diagonals in $\prod_i X^{a_i}$.
\end{defn}

\begin{lem}\label{lem:definition of mu_0}
For each quasi-projective variety $X$, define
\[\mu_0(1-t,-[X]):=1+\sum_{n=1}^\infty[\Sym^n_\Delta X]\cdot t^n.\]
Then $\mu_0$ determines a (finitely determined) power structure on $K_0(\Var_k)$ (which we also denote by $\mu_0$), and $\mu_0(A(t),[X])=\mu_\mr{eff}(A(t),[X])$ for every quasi-projective variety $X$. 
\end{lem}
\begin{proof}
    That $\mu_0$ determines a power structure on $K_0(\Var_k)$ follows from \cite[Proposition 1]{GZLMH:hilbert}. Since $\mu_0$ is a power structure, we have
    \begin{align*}
        \mu_0\left(1+\sum_{n\geq 1}t^n,[X]\right)&=\mu_0((1-t)^{-1},[X])\\
        &=\mu_0(1-t,-[X]),
    \end{align*}
    and $\mu_0(1-t,-[X])=\mu_\mr{eff}(1+\sum_{n\geq 1}t^n,[X])$ by construction. Thus $\mu_0$ and $\mu_\mr{eff}$ agree on pairs of the form $(1-t,-[X])$, so these two formulas determine the same power structure by \cite[Proposition 1]{GZLMH:hilbert}.
\end{proof}

\subsection{A power structure on $K_0^\uh(\Var_k)$}
Since $\chi^c$ is insensitive to universal homeomorphisms, it would suffice for our purposes to work with a power structure on $K_0^\uh(\Var_k)$. Analogous to Lemma~\ref{lem:definition of mu_0}, we can define a power structure $\mu_0^\uh$ on $K_0^\uh(\Var_k)$ by specifying
\[\mu_0^\uh(1-t,-[X]):=1+\sum_{n=1}^\infty[\Sym^n_\Delta X]\cdot t^n\]
for each quasi-projective variety $X$. If we replace $\Sym^n_\Delta X$ by a scheme $Y_n$ such that $[\Sym^n_\Delta X]=[Y_n]$ in $K_0^\uh(\Var_k)$, then the formula
\[\mu(1-t,-[X]):=1+\sum_{n=1}^\infty[Y_n]\cdot t^n\]
determines the same power structure on $K_0^\uh(\Var_k)$ as $\mu_0^\uh$. We list a few universally homeomorphic replacements of $[\Sym^n_\Delta X]$ below.

\begin{prop}\label{prop:universally homeomorphic}
    Let $X$ be a quasi-projective variety over a field $k$. For each $n\geq 1$, the classes in $K_0^\uh(\Var_k)$ of the following $k$-schemes are equal:
    \begin{enumerate}[(i)]
    \item $\Sym_\Delta^n X$.
    \item $\Sym^n X$.
    \item The scheme of degree $n$ divided powers of $X$.
    \end{enumerate}
\end{prop}
\begin{proof}
    The symmetric power and divided power schemes are universally homeomorphic by \cite[Paper III]{Ryd08}, so it suffices to consider $\Sym^n_\Delta X$ and $\Sym^n X$. This is explained in the proof of \cite[Chapter 7, Proposition 1.1.11]{CLNS:MotivicIntegration}, but we will include the relevant details here. 
    
    The stratification of $\Sym^n X$ into orbit types (see e.g.~\cite{Got01,dCM02}) gives us an equality
    \[\Sym^n X=\coprod_{\lambda\in\mc{P}(n)}\Sym^\lambda X,\]
    where $\mc{P}(n)$ is the set of partitions of $n$. Given $\lambda=(\lambda_1,\ldots,\lambda_m)$ (so that $n=\sum_{i=1}^m\lambda_i$), the closed points of the stratum $\Sym^\lambda X$ are 0-cycles $\sum_i\lambda_i[x_i]$, with $x_1,\ldots,x_m$ distinct closed points of $X$. Let $\Delta\subset X^m$ denote the diagonal consisting of $m$-tuples of points in which at least two coincide. Let $S_\lambda\leq S_m$ denote the subgroup consisting of permutations $\sigma$ such that $\lambda_{\sigma(i)}=\lambda_i$ for all $1\leq i\leq m$. This group acts freely on $X^m-\Delta$, and the morphism
    \begin{align*}
        X^m-\Delta&\to\Sym^\lambda X\\
        (x_1,\ldots,x_m)&\mapsto\sum_{i=1}^m\lambda_i[x_i]
    \end{align*}
    factors through the quotient $X^m-\Delta\to(X^m-\Delta)/S_\lambda$. The induced map $(X^m-\Delta)/S_\lambda\to\Sym^\lambda X$ is finite, surjective, and a bijection on geometric points, which implies that $(X^m-\Delta)/S_\lambda\to\Sym^\lambda X$ is a universal homeomorphism.
    
    For our partition $\lambda$, let $a_i$ denote the number of terms $\lambda_j$ such that $\lambda_j=i$ (so that $n=\sum_{i=1}^n ia_i$ and $m=\sum_{i=1}^n a_i$). Let $\prod_{i=1}^n X^{a_i}-\Delta$, where $\Delta$ is again the diagonal consisting of $m$-tuples of points in which at least two coincide. The action of $\prod_i S_{a_i}$ on $\prod_i X^{a_i}-\Delta$ assigns the same weight $\lambda_i$ to the points in the $a_i$-tuple coming from $X^{a_i}-\Delta$; this is identical to the action of $S_\lambda$ on $X^m-\Delta$. In particular, we have an isomorphism
    \[(X^m-\Delta)/S_\lambda\cong\left(\prod_{i=1}^nX^{a_i}-\Delta\right)/\prod_{i=1}^n S_{a_i},\]
    the latter terms being the strata of $\Sym^n_\Delta X$. Thus
    \begin{align*}
        [\Sym^n X]&=\sum_{\lambda\in\mc{P}(n)}[\Sym^n_\lambda X]\\
        &=\sum_{\lambda\in\mc{P}(n)}[(X^m-\Delta)/S_\lambda]\\
        &=\sum_{\lambda\in\mc{P}(n)}\left[\left(\prod_{i=1}^n X^{a_i}-\Delta\right)/\prod_{i=1}^n S_{a_i}\right]\\
        &=[\Sym^n_\Delta X]
    \end{align*}
    in $K_0^\uh(\Var_k)$, as desired.
\end{proof}

\section{A conjectural power structure on $\GW(k)$}\label{sec:power structure gw}
Recall over a field $k$ of characteristic not 2, the motivic Euler characteristic is a ring homomorphism $\chi^c:K_0(\Var_k)\to\GW(k)$ (see Remark~\ref{rem:chi^c valued in gw}). In the complex setting, G\"ottsche's formula arises from the compatibility of the Euler characteristic (with compact support) $\chi:K_0(\Var_\mb{C})\to\mb{Z}$ with the standard power structure on $\mb{Z}$ and Gusein-Zad--Luengo--Melli-Hern\'andez's power structure on $K_0(\Var_\mb{C})$. 

We have given an analogous Euler characteristic $\chi^c$ and recounted a power structure on $K_0(\Var_k)$. Over any field $k$ with $\Char{k}\neq 2$ (and conditional on Conjecture~\ref{conj:main}), we will use these constructions to induce a power structure on $\GW(k)$ such that
\[\chi^c(A(t)^{M})=\chi^c(A(t))^{\chi^c(M)}\]
for all $M\in K_0(\Var_k)$ and $A(t)\in 1+t\cdot K_0(\Var_k)\dlb t\drb$. The reason that we assume $\Char{k}\neq 2$ is so that the motivic Euler characteristic is a ring homomorphism $\chi^c:K_0(\Var_k)\to\GW(k)$ (see Corollary~\ref{cor:gw inside gw[1/p]}). In Proposition~\ref{prop:chi^c surjective}, we will prove that $\chi^c$ surjects onto $\GW(k)$. The next subsection will explain the role of surjective ring homomorphisms with respect to power structures.

\subsection{Pre-power structures}
Recall that a power structure on a ring $R$ is a function $(1+t\cdot R\dlb t\drb)\times R\to 1+t\cdot R\dlb t\drb$ that satisfies the formal properties usually associated with exponentiation (see Definition~\ref{def:power structure}). Rather than defining the values of such a function for each $f(t)\in 1+t\cdot R\dlb t\drb$, it will be convenient to work with a simpler subset of $1+t\cdot R\dlb t\drb$. This leads us to the notion of a \textit{pre-power structure}.

\begin{defn}\label{def:pre-power structure}
Let $R$ be a ring. Let $S_R:=\{1-t^i:i\geq 0\}\subset 1+t\cdot R\dlb t\drb$. A \textit{pre-power structure} on $R$ is a function
\[S_R\times R\to 1+t\cdot R\dlb t\drb,\]
denoted by $((1-t^i),r)\mapsto (1-t^i)^r$, satisfying the following properties for all $r,s\in R$:
\begin{enumerate}[(i)]
\item $(1-t)^{-1}=\sum_{i=0}^\infty t^i$.
\item $(1-t)^{-a}=1+at+o(t^2)$.
\item $(1-t)^{-(a+b)}=(1-t)^{-a}\cdot(1-t)^{-b}$.
\item If $(1-t)^a=f(t)$, then $(1-t^i)^a=f(t^i)$ for all $i\geq 1$.
\end{enumerate}
\end{defn}

Note that any power structure restricts to a pre-power structure.

\begin{prop}
Let $R$ be a ring with a power structure $\mu:(1+t\cdot R\dlb t\drb)\times R\to 1+t\cdot R\dlb t\drb$. Then $\mu|_{S_R\times R}$ is a pre-power structure on $R$.
\end{prop}
\begin{proof}
Properties (ii), (iii), and (iv) of Definition~\ref{def:pre-power structure} are special cases of properties (vi), (iv), and (vii), respectively, of Definition~\ref{def:power structure}. To deduce property (i) of Definition~\ref{def:pre-power structure}, note that
\begin{align*}
    1 &= (1-t)^0\\
    &= (1-t)^1(1-t)^{-1}.
\end{align*}
We can thus derive the formula $(1-t)^{-1}=\sum_{i=0}^\infty t^i$ by inductively computing $1=(1-t)(\sum_{i=0}^\infty t^i)$.
\end{proof}

The benefits of working with a pre-power structure are twofold: a pre-power structure uniquely extends to a power structure \cite[Proposition 1]{GZLMH:hilbert}, and a ring homomorphism $\vphi:R_1\to R_2$ respects given power structures $\mu_i$ on $R_i$ if and only if $\vphi$ respects the pre-power structures $\mu_i|_{S_{R_i}\times R_i}$ \cite[Proposition 2]{GZLMH:hilbert}. In the following lemma, we will utilize these results to show when a surjective ring homomorphism sends a power structure on the source to a compatible power structure on the target.

Before stating the lemma, recall that a ring homomorphism $\vphi:R_1\to R_2$ induces a ring homomorphism $R_1\dlb t\drb\to R_2\dlb t\drb$. The ring structure on $R_i\dlb t\drb$ is given by formal addition and multiplication of power series, with the ring structure on $R_i$ governing the addition and multiplication of coefficients.

\begin{notn}\label{notn:ker chi mult closed}
Let $\vphi:R_1\to R_2$ be a ring homomorphism. Let $1+t\cdot\ker\vphi\dlb t\drb$ denote the subset of $1+t\cdot R_1\dlb t\drb$ consisting of power series whose positive degree coefficients are all elements of $\ker\vphi\subseteq R_1$. Note that $1+t\cdot\ker\vphi\dlb t\drb$ is a multiplicatively closed set, as $\vphi$ is a ring homomorphism and all positive degree coefficients in a product of elements of $1+t\cdot\ker\vphi\dlb t\drb$ consist of sums of products of elements of $\ker\vphi$.
\end{notn}

\begin{lem}\label{lem:induced power structure}
Let $\vphi:R_1\to R_2$ be a surjective ring homomorphism. Let $\mu_1$ be a power structure on $R_1$. Assume that $\mu_1(1-t,r)\in 1+t\cdot\ker\vphi\dlb t\drb$ for all $r\in\ker\vphi$. Then there exists a unique power structure $\mu_2$ on $R_2$ such that $\vphi(\mu_1(A(t),r))=\mu_2(\vphi(A(t)),\vphi(r))$ for all $A(t)\in 1+t\cdot R_1\dlb t\drb$ and $r\in R_1$.
\end{lem}
\begin{proof}
Since $\vphi$ is surjective, $\vphi^{-1}(s)$ is non-empty for each $s\in R_2$. We begin by defining $\mu_2(1-t^i,s):=\vphi(\mu_1(1-t^i,\vphi^{-1}(s)))$ for each $i\geq 1$ and $s\in R_2$. To see that this is well-defined, note that given $a,b\in\vphi^{-1}(s)$, there exists $r\in\ker\vphi$ such that $a=b+r$. Thus $\mu_1(1-t^i,a)=\mu_1(1-t^i,b)\cdot\mu_1(1-t^i,r)$, so
\begin{align*}
    \vphi(\mu_1(1-t^i,a))&=\vphi(\mu_1(1-t^i,b)\cdot\mu_1(1-t^i,r))\\
    &=\vphi(\mu_1(1-t^i,b))\cdot\vphi(\mu_1(1-t^i,r)).
\end{align*}
Since $\mu_1$ is a power structure, we have $\mu_1(1-t^i,r)=\mu_1(1-t,r)|_{t\mapsto t^i}$. That is, the coefficients in positive degree of $\mu_1(1-t^i,r)$ belong to the same set as the coefficients in positive degree of $\mu_1(1-t,r)$. In particular, $\mu_1(1-t,r)\in 1+t\cdot\ker\vphi\dlb t\drb$ implies that $\mu_1(1-t^i,r)\in 1+t\cdot\ker\vphi\dlb t\drb$ for all $i\geq 1$. Thus $\vphi(\mu_1(1-t^i,r))=1$ for all $r\in\ker\vphi$, so $\mu_2(1-t^i,\vphi^{-1}(s))$ is well-defined.

This is enough to yield a pre-power structure on $R_2$. Indeed:
\begin{enumerate}[(i)]
\item We have $\vphi(1_{R_1})=1_{R_2}$ and $\vphi(-1_{R_1})=-1_{R_2}$, as is the case for any ring homomorphism. Since $\mu_1(1-t,-1)=\sum_{i=0}^\infty t^i$, we have
\begin{align*}
    \mu_2(1-t,-1)&=\vphi(\mu_1(1-t,-1))\\
    &=\vphi(\textstyle\sum_{i=0}^\infty t^i)\\
    &=\textstyle\sum_{i=0}^\infty t^i.
\end{align*}
\item Since $\mu_1(1-t,-a)=1+at+o(t^2)$, we have
\begin{align*}
    \mu_2(1-t,-s)&=\vphi(\mu_1(1-t,-\vphi^{-1}(s)))\\
    &=\vphi(1+\vphi^{-1}(s)t+o(t^2))\\
    &=1+st+o(t^2).
\end{align*}
\item Since $\mu_1(1-t,-(a+b))=\mu_1(1-t,-a)\cdot\mu_1(1-t,-b)$, we have
\begin{align*}
    \mu_2(1-t,-(c+d))&=\vphi(\mu_1(1-t,-(\vphi^{-1}(c)+\vphi^{-1}(d))))\\
    &=\vphi(\mu_1(1-t,-\vphi^{-1}(c)))\cdot\vphi(\mu_1(1-t,-\vphi^{-1}(d)))\\
    &=\mu_2(1-t,-c)\cdot\mu_2(1-t,-d).
\end{align*}
\item Let $f(t):=\mu_2(1-t,s)$. Then $f(t)=\vphi(g(t))$, where $g(t)=\mu_1(1-t,\vphi^{-1}(s))$. Since $\mu_1$ is a power structure on $R_1$, we have $g(t^i)=\mu_1(1-t^i,\vphi^{-1}(s))$. Because $\vphi$ acts on power series by acting on their coefficients, we have $\vphi(g(t^i))=f(t^i)$. Thus $f(t^i)=\mu_2(1-t^i,s)$, as desired.
\end{enumerate}
By \cite[Proposition 1]{GZLMH:hilbert}, $\mu_2$ extends to a power structure on $R_2$. By construction, $\vphi$ respects the pre-power structures underlying $\mu_1$ and $\mu_2$, so \cite[Proposition 2]{GZLMH:hilbert} implies that $\vphi(\mu_1(A(t),r))=\mu_2(\vphi(A(t)),\vphi(r))$ for all $A(t)\in 1+t\cdot R_1\dlb t\drb$ and $r\in R_1$.
\end{proof}

\subsection{Symmetric powers of $\ker\chi^c$}
Let $k$ be a field with odd exponential characteristic (i.e.~$\Char{k}\neq 2$). Our goal is to induce a power structure on $\GW(k)$ from the power structure on $K_0(\Var_k)$ by way of Lemma~\ref{lem:induced power structure}. In order to do this, we need to show that $\chi^c:K_0(\Var_k)\to\GW(k)$ is surjective. Remarkably, we do not need any virtual classes in $K_0(\Var_k)$ to prove the surjectivity of $\chi^c$.

\begin{prop}\label{prop:chi^c surjective}
Let $k$ be a field with $\Char{k}\neq 2$. Then for each $q\in\GW(k)$, there exists a variety $X_q\in\Var_k$ such that $\chi^c(X_q)=q$. In particular,
\[\chi^c:K_0(\Var_k)\to\GW(k)\]
is surjective.
\end{prop}
\begin{proof}
Since $\GW(k)$ is additively generated by isomorphism classes of rank $\pm 1$ forms \cite[II Theorem 4.1]{Lam:quadratic-forms}, there exist  $a_1,\ldots,a_n,b_1,\ldots,b_m\in k^\times$ such that $q=\sum_{i=1}^n\langle a_i\rangle-\sum_{j=1}^m\langle b_j\rangle$. Since $\chi^c$ is additive on disjoint unions, it suffices to find, for each $a\in k^\times$, varieties $X_{\langle a\rangle}$ and $X_{-\langle a\rangle}$ such that $\chi^c(X_{\langle a\rangle})=\langle a\rangle$ and $\chi^c(X_{-\langle a\rangle})=-\langle a\rangle$. 

Let $a\in k^\times/(k^\times)^2$. Since $\Char{k}\neq 2$, the degree 2 extension $k(\sqrt{a})/k$ is separable. We can thus apply Example~\ref{ex:motivic = cat} (i) to compute $\chi^c(\Spec{k(\sqrt{a})})=\Tr_{k(\sqrt{a})/k}\langle 1\rangle=\langle 2\rangle+\langle 2a\rangle$. We now proceed in two cases.
\begin{enumerate}[(i)]
\item First, assume that $2\in k$ is not a square. Let $x_a:\Spec{k(\sqrt{a})}\to\mb{P}^1_k$ denote $\Spec{k(\sqrt{a})}$ as a degree $2$ closed point in $\mb{P}^1_k$. Then
\begin{align*}
    \chi^c(\mb{P}^1_k-\{x_2\})&=\mb{H}-\langle 2\rangle-\langle 4\rangle\\
    &=\langle 2\rangle+\langle -2\rangle-\langle 2\rangle-\langle 1\rangle\\
    &=\langle-2\rangle-\langle 1\rangle.
\end{align*}
Let $X_{\langle -2\rangle}=(\mb{P}^1_k-\{x_2\})\amalg\Spec{k}$, which satisfies $\chi^c(X_{\langle -2\rangle})=\langle-2\rangle$ by Proposition~\ref{prop:chi additive} and the above calculation. Let $X_{\langle -1\rangle}=\mb{G}_m\amalg\Spec{k}$, which satisfies $\chi^c(X_{\langle-1\rangle})=\langle-1\rangle$ by Proposition~\ref{prop:chi additive} and Corollary~\ref{cor:punctured affine}. Then $\chi^c(X_{\langle -2\rangle}\times X_{\langle -1\rangle})=\langle 2\rangle$ by Proposition~\ref{prop:chi mult}. We then have
\begin{align*}
\chi^c(X_{\langle -2\rangle}\times(\mb{P}^1_k-\{x_a\}))&=\langle -2\rangle(\langle-2\rangle-\langle 2a\rangle)\\
&=\langle 1\rangle-\langle -a\rangle.
\end{align*}
Since $0\in\mb{P}^1_k-\{x_a\}$ and $0\in(\mb{P}^1_k-\{x_2\})\amalg\Spec{k}$ are both $k$-rational, we can define $X_{-\langle -a\rangle}:=X_{\langle -2\rangle}\times(\mb{P}^1_k-\{x_a\})-\{(0,0)\}$, which satisfies $\chi^c(X_{-\langle-a\rangle})=-\langle-a\rangle$. We then have $\chi^c(X_{\langle-1\rangle}\times X_{-\langle-a\rangle})=-\langle a\rangle$ by Proposition~\ref{prop:chi mult}. We also have
\begin{align*}
\chi^c(X_{-\langle-a\rangle}\amalg\mb{P}^1_k)&=-\langle-a\rangle+\langle a\rangle+\langle-a\rangle\\
&=\langle a\rangle,
\end{align*}
as desired.
\item Second, assume that $2\in k$ is a square. Then $\chi^c(\Spec{k(\sqrt{a})})=\langle 1\rangle+\langle a\rangle$, so $\chi^c(\Spec{k(\sqrt{a})}\amalg(\mb{A}^2_k-\{0,1\}))=\langle a\rangle$ by Proposition~\ref{prop:chi additive} and Corollary~\ref{cor:punctured affine}. Similarly, since $x_a:\Spec{k(\sqrt{a})}\to\mb{A}^1_k-\{0\}$ is a closed point, we have $\chi^c((\mb{A}^1_k-\{0,x_a\})\amalg\Spec{k})=-\langle a\rangle$.
\end{enumerate}
In either case, we have constructed (for each $a\in k^\times/(k^\times)^2$) the desired varieties $X_{\langle a\rangle}$ and $X_{-\langle a\rangle}$ with $\chi^c(X_{\langle a\rangle})=\langle a\rangle$ and $\chi^c(X_{-\langle a\rangle})=-\langle a\rangle$.
\end{proof}

Now that we know that $\chi^c:K_0(\Var_k)\to\GW(k)$ is surjective, we can obtain a power structure on $\GW(k)$ that is compatible with $\chi^c$ and the power structure on $K_0(\Var_k)$ by proving $(1-t)^{\ker\chi^c}\in 1+t\cdot\ker\chi^c\dlb t\drb$. We will start by considering virtual classes of quasi-projective varieties.

\begin{prop}\label{prop:virtual classes in kernel}
Let $X$ be a quasi-projective $k$-variety. Assume that $[X]\in\ker\chi^c$. Then $(1-t)^{-[X]}\in 1+t\cdot\ker\chi^c\dlb t\drb$ if and only if $[\Sym^n X]\in\ker\chi^c$ for all $n\geq 1$.
\end{prop}
\begin{proof}
By definition of the power structure on $K_0(\Var_k)$ (Lemma~\ref{lem:definition of mu_0}), we have $(1-t)^{-[X]}=1+\sum_{n\geq 1}[\Sym_\Delta^n X]\cdot t^n$. Thus $(1-t)^{-[X]}\in 1+t\cdot\ker\chi^c\dlb t\drb$ if and only  if $[\Sym_\Delta^n X]\in\ker\chi^c$ for all $n\geq 1$. Since $\Sym_\Delta^n X$ and $\Sym^n X$ are equal in $K_0^{\rm{uh}}(\rm{Var}_k)$ for all $n\geq 1$ (Proposition~\ref{prop:universally homeomorphic}), the desired result follows from Corollary~\ref{cor:chi universal homeo}.
\end{proof}

This leads us to Conjecture~\ref{conj:main}, which we restate here.

\begin{conj}\label{conj:sym power}
If $X$ is a quasi-projective $k$-variety with $[X]\in\ker\chi^c$, then $[\Sym^n X]\in\ker\chi^c$ for all $n\geq 1$ (equivalently, $[\Sym^n_{\Delta}X] \in \ker\chi^c$ for all $n \geq 1$).
\end{conj}

Conditional on Conjecture~\ref{conj:sym power}, we can now consider classes in $\ker\chi^c$ that are purely sums of genuine classes or purely sums of virtual classes.

\begin{cor}\label{cor:sums of classes in kernel}
Assume Conjecture~\ref{conj:sym power}. Let $X_1,\ldots,X_n$ be quasi-projective $k$-varieties. Let $M=\pm\sum_{i=1}^n[X_i]$. Assume that $M\in\ker\chi^c$. Then $(1-t)^M\in 1+t\cdot\ker\chi^c\dlb t\drb$.
\end{cor}
\begin{proof}
Let $X=\coprod_{i=1}^nX_i$, so that $\sum_{i=1}^n[X_i]=[X]$. Since each $X_i$ is quasi-projective, so is $X$. In the case that $-[X]=M$, the desired result is now given in Proposition~\ref{prop:virtual classes in kernel}.

Suppose instead that $[X]=M$.
Properties (i) and (iii) of Definition~\ref{def:pre-power structure} imply that $(1-t)^M\cdot(1-t)^{-M}=1$. Let $(1-t)^M=1+\sum_{i\geq 1}A_it^i$ and $(1-t)^{-M}=1+\sum_{i\geq 1}B_it^i$, where $A_i,B_i\in K_0(
\Var_k)$. Set $A_0=B_0=1$. With this notation, we have
\[1=1+\sum_{\ell\geq 1}\left(\sum_{\substack{i,j\geq 0\\i+j=\ell}}A_iB_j\right)t^\ell.\]
The first paragraph of this proof implies that $B_i\in\ker\chi^c$ for all $i$. We now proceed by induction. For $\ell=1$, we have $A_1+B_1=0$, so $B_1\in\ker\chi^c$ implies that $A_1\in\ker\chi^c$. Now assume that $A_1,\ldots,A_n\in\ker\chi^c$. Since $0=A_{n+1}+\sum_{i=0}^nA_iB_{n+1-i}$ and $A_iB_{n+1-i}\in\ker\chi^c$ for $0\leq i\leq n$, we find that $A_{n+1}\in\ker\chi^c$ as well. It follows that $(1-t)^M\in 1+t\cdot\ker\chi^c\dlb t\drb$, as desired.
\end{proof}

Finally, we can consider arbitrary classes in $\ker\chi^c$.

\begin{prop}\label{prop:kernel of chi^c}
Assume Conjecture~\ref{conj:sym power}. If $M\in K_0(\Var_k)$ with $\chi^c(M)=0$, then $(1-t)^M\in 1+t\cdot\ker\chi^c\dlb t\drb$.
\end{prop}
\begin{proof}
Since $K_0(\Var_k)$ is generated as a group by classes of quasi-projective $k$-varieties \cite[Chapter 2, Corollary 2.6.6(a)]{CLNS:MotivicIntegration}, there exist quasi projective $k$-varieties $\{X_i\}_{i=1}^a$ and $\{Y_j\}_{j=1}^b$ such that $M=\sum_{i=1}^a [X_i]-\sum_{j=1}^b [Y_j]$. Let $X=\coprod_{i=1}^a X_i$ and $Y=\coprod_{j=1}^b Y_j$, so that $M=[X]-[Y]$.

Since $X$ is quasi-projective over a noetherian base, we can partition $X$ into a finite set of locally closed affine subvarieties $\{A_i\}_{i=1}^c$, so that $[X]=\sum_{i=1}^c[A_i]$. Each of these affine varieties $A_i$ is a subvariety of $\mb{A}^{n_i}$, where the dimension $n_i\geq 1$ depends on $i$. Using any closed embedding $\mb{A}^{n_i}\subset\mb{A}^{2n_i}-\{(1,\ldots,1)\}$ (for example, given by the vanishing of the first $n$ variables),
we find that $A_i$ also occurs as a subvariety of $B_i:=\mb{A}^{2n_i}-\{(1,\ldots,1)\}$. In particular, $B_i-A_i$ is a quasi-projective $k$-variety with $[B_i-A_i]=[B_i]-[A_i]$.

Let $B:=\coprod_{i=1}^c B_i$ and $C:=\coprod_{i=1}^c(B_i-A_i)$. By Corollary~\ref{cor:punctured affine} (iii) (after translating the origin to $(1,\ldots,1)$), we have $[B_i]\in\ker\chi^c$ and thus $[B]\in\ker\chi^c$. Moreover,
\begin{align*}
    M-\sum_{i=1}^c[B_i]&=-[Y]+\sum_{i=1}^c[A_i]-\sum_{i=1}^c[B_i]\\
    &=-[Y]-\sum_{i=1}^c[B_i-A_i]\\
    &=-[C\amalg Y].
\end{align*}
Since $M,[B]\in\ker\chi^c$, we have $[C\amalg Y]\in\ker\chi^c$. In particular, we have $(1-t)^M=(1-t)^{-[C\amalg Y]}\cdot(1-t)^{[B]}$ with $[C\amalg Y],[B]\in\ker\chi^c$. By Corollary~\ref{cor:sums of classes in kernel}, it follows that $(1-t)^{[B]},(1-t)^{-[C\amalg Y]}\in 1+t\cdot\ker\chi^c\dlb t\drb$. We conclude by noting that $1+t\cdot\ker\chi^c\dlb t\drb$ is multiplicatively closed (see Notation~\ref{notn:ker chi mult closed}).
\end{proof}

We are now ready to prove Theorem~\ref{thm:power structure}, which we restate here.

\begin{thm}\label{thm:power structure (second appearance)}
Let $k$ be a field of characteristic not 2. Let
\[\mu_0:(1+t\cdot K_0(\Var_k)\dlb t\drb)\times K_0(\Var_k)\to 1+t\cdot K_0(\Var_k)\dlb t\drb\]
be the power structure defined in Section~\ref{sec:power structure K0}. Then Conjecture~\ref{conj:main} is true if and only if there exists a power structure
\[\mu_\GW:(1+t\cdot\GW(k)\dlb t\drb)\times\GW(k)\to 1+t\cdot \GW(k)\dlb t\drb\]
such that
\[\chi^c(\mu_0(A(t),M))=\mu_\GW(\chi^c(A(t)),\chi^c(M))\]
for all $A(t)\in 1+t\cdot K_0(\Var_k)\dlb t\drb$ and $M\in K_0(\Var_k)$.
\end{thm}
\begin{proof}
First assume that Conjecture~\ref{conj:main} is true. Since $\chi^c:K_0(\Var_k)\to\GW(k)$ is surjective (Proposition~\ref{prop:chi^c surjective}) and $\mu_0(1-t,M)\in 1+t\cdot\ker\chi^c\dlb t\drb$ for all $M\in\ker\chi^c$ (Proposition~\ref{prop:kernel of chi^c}), the result follows from Lemma~\ref{lem:induced power structure}.

Now assume that the power structure $\mu_\GW$ exists. The idea is to use Proposition~\ref{prop:virtual classes in kernel}: given any quasi-projective variety $X$ with $\chi^c(X)=0$, we have
\begin{align*}
    1&=\mu_\GW(1-t,0)\\
    &=\mu_\GW(\chi^c(1-t),\chi^c(-[X]))\\
    &=\chi^c(\mu_0((1-t),-[X]))\\
    &=1+\sum_{n\geq 1}\chi^c(\Sym_\Delta^n X)\cdot t^n\\
    &=1+\sum_{n\geq 1}\chi^c(\Sym^n X)\cdot t^n.
\end{align*}
It follows that $\chi^c(\Sym^n X)=0$ for all $n\geq 1$, which implies Conjecture~\ref{conj:main}.
\end{proof}

\section{Enriching G\"ottsche's formula}\label{sec:gottsche}
Now that we have developed a conjectural power structure on $\GW(k)$ that is compatible with the power structure on $K_0(\Var_k)$ via $\chi^c$, we are prepared to return to the Yau--Zaslow formula. The relevant generating function comes from G\"ottsche's formula \cite{Got90}, which states that for a smooth projective surface $X$ over an algebraically closed field, the Hilbert schemes $\Hilb^gX$ satisfy
\begin{equation}\label{eq:gottsche}
\sum_{g\geq 0}\chi(\Hilb^gX)\cdot t^g=\prod_{n\geq 1}(1-t^n)^{-\chi(X)}.
\end{equation}
Equation~\ref{eq:gottsche} admits a generalization in terms of the power structure on $K_0(\Var_k)$. Given a scheme $Y$ and a point $y\in Y$, let $\Hilb^m_yY$ denote the Hilbert scheme of length $m$, dimension 0 subschemes of $Y$ that are concentrated at the point $y$. Then for any smooth quasi-projective variety $X$ of dimension $d$, there is an equality 
\begin{equation}\label{eq:power gottsche}
\sum_{g\geq 0}[\Hilb^gX]\cdot t^g=\big(\sum_{m\geq 0}[\Hilb^m_0(\mb{A}^d)]\cdot t^m\big)^{[X]}
\end{equation}
in $K_0^\uh(\Var_k)\dlb t\drb$ \cite[Theorem 1]{GZLMH:hilbert}. (This result is stated in $K_0(\Var_\mb{C})$ in \textit{loc.~cit.}, but one can repeat the same argument verbatim over any field with the power structure on $K_0^\uh(\Var_k)$ given in Section~\ref{sec:power structure K0}.)

To deduce Equation~\ref{eq:gottsche} from Equation~\ref{eq:power gottsche}, one uses the compatibility of the power structures on $K_0^\uh(\Var_{\overline{k}})$ and $\mb{Z}$ via the Euler characteristic, as well as Ellingsrud--Str\o mme's computation of $\chi(\Hilb^m_0(\mb{A}^2))$ as the partition number of $m$ \cite{ES87}:
\[\sum_{m\geq 0}\chi(\Hilb^m_0(\mb{A}^2))\cdot t^m=\prod_{n\geq 1}(1-t^n)^{-1}.\]
To enrich G\"ottsche's formula, we therefore need Conjecture~\ref{conj:main} to hold (which gives the necessary power structure on $\GW(k)$) and a computation of $\chi^c(\Hilb^m_0(\mb{A}^2))$.

\subsection{Computing $\chi^c(\Hilb^m_0(\mb{A}^2))$}
We now set out to give a generating series for the motivic Euler characteristic of the local punctual Hilbert scheme $\Hilb^m_0(\mb{A}^2)$. We will show that the cell decomposition of $\Hilb^m_0(\mb{A}^2_\mb{C})$ given in \cite{ES88} in fact holds over arbitrary fields. This will allow us to compute $\chi^c(\Hilb^m_0(\mb{A}^2_k))$ in terms of $\chi^c(\mb{A}^i)=\langle -1\rangle^i$. As a result, we will find that $\chi^c(\Hilb^m_0(\mb{A}^2_k))$ is determined entirely by its values for $k=\mb{C}$ and $\mb{R}$, which were computed by Ellingsrud--Str\o mme \cite{ES87} and Kharlamov--R\u{a}sdeaconu \cite{KR15}, respectively.

\begin{defn}
A \textit{cellular decomposition} of a scheme $X$ is a chain $X=X_n\supset X_{n-1}\supset\cdots\supset X_1\supset X_0=\varnothing$ of closed subschemes such that each $X_i-X_{i-1}$ is a disjoint union of locally closed subschemes isomorphic to affine spaces.
\end{defn}

Given such a cellular decomposition of $X$, note that $X=[X_n-X_{n-1}]+[X_{n-1}]$ in $K_0(\Var_k)$. Inducting, we find that $[X]=\sum_{i=1}^n[X_i-X_{i-1}]$, as $[X_0]=0$. We can then further decompose each class $[X_i-X_{i-1}]$ by expressing $X_i-X_{i-1}$ as a disjoint union of affine spaces.

\begin{lem}\label{lem:cellular decomposition}
Let $X$ be a scheme over a field of characteristic not 2. If $X$ admits a cellular decomposition over $k$, then $\chi^c(X)=a\langle 1\rangle+b\langle -1\rangle$ for some $a,b\in\mb{N}\cup\{0\}$.
\end{lem}
\begin{proof}
If $X$ admits a cellular decomposition, then
\begin{align*}
    [X]&=\sum_{i=1}^n[X_i-X_{i-1}]\\
    &=\sum_{i=1}^n\sum_{j=1}^{m_i}[\mb{A}^{d(i,j)}_k],
\end{align*}
where $d(i,j)\in\mb{N}\cup\{0\}$. Applying $\chi^c$ and recalling that $\chi^c(\mb{A}^{d(i,j)})=\langle\pm 1\rangle$ proves the claim.
\end{proof}

Since the element $a\langle 1\rangle+b\langle -1\rangle\in\GW(k)$ is completely determined by its rank and signature, Lemma~\ref{lem:cellular decomposition} allows us to compute the motivic Euler characteristic of schemes admitting a cellular decomposition purely in terms of topological data, namely the Euler characteristic of the complex and real loci. We now apply this strategy to $\Hilb^m_0(\mb{A}^2)$.

\begin{prop}\label{prop:cell decomposition of Hilb0}
If $k$ is a field, then the scheme $\Hilb^m_0(\mb{A}^2_k)$ admits a cellular decomposition.
\end{prop}
\begin{proof}
Although the cellular decompositions of $\Hilb^m(\mb{A}^2)$ and $\Hilb^m_0(\mb{A}^2)$ in \cite[\S3]{ES88} are only stated over $\mb{C}$ in loc.~cit., the requisite Bia\l ynicki-Birula theory \cite{BB73,BB76} in fact holds over arbitrary fields (see e.g.~\cite[\S1, \S2, and Example 7.9]{JS19}). In particular, $\Hilb^m_0(\mb{A}^2_k)$ admits a cellular decomposition.
\end{proof}

As a result, we can deduce the generating function of $\chi^c(\Hilb^m_0(\mb{A}^2_k))$ from the complex and real settings.

\begin{cor}\label{cor:generating function for local hilbert}
Let $k$ be a field of characteristic not 2. Then
\begin{equation}\label{eq:motivic gottsche for local hilbert}
\sum_{m\geq 0}\chi^c(\Hilb^m_0(\mb{A}^2_k))\cdot t^m=\prod_{n\geq 1}(1-\langle -1\rangle^{n-1}\cdot t^n)^{-1}.
\end{equation}
\end{cor}
\begin{proof}
Lemma~\ref{lem:cellular decomposition} and Proposition~\ref{prop:cell decomposition of Hilb0} give us
\[\chi^c(\Hilb^m_0(\mb{A}^2_k))=\frac{e_\mb{C}+e_\mb{R}}{2}\langle 1\rangle+\frac{e_\mb{C}-e_\mb{R}}{2}\langle -1\rangle,\]
where $e_\mb{C}=\chi(\Hilb^m_0(\mb{A}^2_\mb{C}))$ and $e_\mb{R}=\chi(\Hilb^m_0(\mb{A}^2_\mb{R}))$. We conclude by applying \cite[(1.1) Theorem (iv)]{ES87} (together with the partition generating series) and \cite[Proposition~3.1]{KR15}, which compute
\begin{align}
    \sum_{m\geq 0}\chi(\Hilb^m_0(\mb{A}^2_\mb{C}))\cdot t^m&=\prod_{n\geq 1}(1-t^n)^{-1},\label{eq:complex prod}\\
    \sum_{m\geq 0}\chi(\Hilb^m_0(\mb{A}^2_\mb{R}))\cdot t^m&=\prod_{n\geq 1}(1 + (-t)^n)^{-1} = \prod_{n\geq 1}(1 - (-1)^{n-1}t^n)^{-1} .\label{eq:real prod}
\end{align}
We therefore need a series in $\mb{Z}\{\langle 1\rangle,\langle -1\rangle\}\dlb t\drb$ whose rank and signature recover Equations~\ref{eq:complex prod} and~\ref{eq:real prod}, respectively (cf.~Propositions~\ref{prop:rank = etale} and~\ref{prop:sign = etale of real locus}). This is given by the product $\prod_{n\geq 1}(1-q_nt^n)^{-1}$, where $q_n=a\langle 1\rangle+b\langle -1\rangle$ has rank 1 and signature $(-1)^{n-1}$. These criteria determine the form $q_n=\langle -1\rangle^{n-1}$.
\end{proof}

\begin{rem}\label{rem:hilb_zeta}
Equivalently, we could directly compute the following generating series identity over $K_0(\rm{Var}_k)$: 
\begin{equation}\label{eqn:partition_punctual}
\sum_{m\geq 0}[\Hilb^m_0(\mb{A}^2_k)]\cdot t^m=\prod_{n\geq 1}(1-\mb{L}^{n-1}\cdot t^n)^{-1}.
\end{equation}
Indeed, the cells in the cell decomposition of \cite{ES87} are indexed by partitions of $n$. For each partition $\lambda \vdash n$, the cell $C_\lambda$ has dimension $n - l(\lambda)$ where $l$ is the number of parts of the partition. Standard generating function arguments for partitions gives (\ref{eqn:partition_punctual}) (see \cite[Proposition 2.1]{hilb_zeta_nonreduced} for more details). Applying $\chi^c$ then yields the corollary. 
\end{rem}

Combining Corollary~\ref{cor:generating function for local hilbert} with Equation~\ref{eq:power gottsche} and Theorem~\ref{thm:power structure}, we get a conditional enrichment of G\"ottsche's formula.

\begin{thm}\label{thm:enriched gottsche}
Let $k$ be a field of characteristic not 2. Assume Conjecture~\ref{conj:main}. Then for any smooth quasi-projective $k$-surface $X$, we have
\begin{equation*}
\sum_{g\geq 0}\chi^c(\Hilb^g X)\cdot t^g=\prod_{n\geq 1}(1-\langle -1\rangle^{n-1}\cdot t^n)^{-\chi^c(X)}
\end{equation*}
in $\GW(k)\dlb t\drb$.
\end{thm}
\begin{proof}
Theorem~\ref{thm:power structure} and Equation~\ref{eq:power gottsche} imply that
\[\sum_{g\geq 0}\chi^c(\Hilb^gX)\cdot t^g=\big(\sum_{m\geq 0}\chi^c(\Hilb^m_0(\mb{A}^2))\cdot t^m\big)^{\chi^c(X)}.\]
By Corollary~\ref{cor:generating function for local hilbert}, we have
\[\big(\sum_{m\geq 0}\chi^c(\Hilb^m_0(\mb{A}^2))\cdot t^m\big)^{\chi^c(X)}=\prod_{n\geq 1}(1-\langle -1\rangle^{n-1}\cdot t^n)^{-\chi^c(X)},\]
as desired.
\end{proof}

Using a theorem of Totaro, we can give an alternative formulation of the generating series of $\chi^c(\Hilb^g X)$ that does not depend on Conjecture~\ref{conj:main}. This result was stated in the introduction as Theorem~\ref{thm:unconditional gottsche}:

\begin{thm}\label{thm:unconditional gottsche formula (second)}
Let $k$ be a field of characteristic not 2. Then for any smooth quasi-projective $k$-surface $X$, we have
\begin{equation}\label{eq:hilbert vs sym}
\sum_{g\geq 0}\chi^c(\Hilb^g X)\cdot t^g=\prod_{n\geq 1}\left(1+\sum_{m\geq 1}\langle -1\rangle^{m(n-1)}\chi^c(\Sym^m X)\cdot t^{mn}\right)
\end{equation}
in $\GW(k)\dlb t\drb$.
\end{thm}
\begin{proof}
Since the power structure on $K_0^\uh(\Var_k)$ is finitely determined, one can compute
\begin{align*}
    (1-\mb{L}^{n-1}\cdot t^n)^{-[X]}&=(1-t^n)^{-[\mb{A}^{n-1}\times X]}\\
    &=1+\sum_{m\geq 1}[\Sym^m(\mb{A}^{n-1}\times X)]\cdot t^{mn}.
\end{align*}
By Equations~\ref{eq:power gottsche} and~\ref{eqn:partition_punctual}, we thus have an equality 
\begin{align*}
\sum_{g\geq 0}[\Hilb^g X]\cdot t^g&=\prod_{n\geq 1}(1-\mb{L}^{n-1}\cdot t^n)^{-[X]}\\
&=\prod_{n\geq 1}\left(1+\sum_{m\geq 1}[\Sym^m(\mb{A}^{n-1}\times X)]\cdot t^{mn}\right)
\end{align*}
in $K_0^\uh(\Var_k)$. Totaro proved that
\[[\Sym^m(\mb{A}^{n-1}\times X)]=[\Sym^m X]\cdot\mb{L}^{m(n-1)}\]
in $K_0^\uh(\Var_k)$ (see e.g.~\cite[Chapter 7, Proposition 1.1.11]{CLNS:MotivicIntegration}), so
\[\prod_{n\geq 1}\left(1+\sum_{m\geq 1}[\Sym^m(\mb{A}^{n-1}\times X)]\cdot t^{mn}\right)=\prod_{n\geq 1}\left(1+\sum_{m\geq 1}[\Sym^m X]\cdot\mb{L}^{m(n-1)}\cdot t^{mn}\right).\]
The result now follows from the fact that $\chi^c:K_0^\uh(\Var_k)\dlb t\drb\to\GW(k)\dlb t\drb$ is a ring homomorphism.
\end{proof}

\subsection{Computing $\chi^c(\Hilb^g X)$ directly}
The motivation behind Conjecture~\ref{conj:main} is to enable a conceptual proof of Theorem~\ref{thm:enriched gottsche} in terms of power structures. However, we can actually compute $\chi^c(\Hilb^g X)$ directly for any smooth projective surface $X$ even without assuming Conjecture~\ref{conj:main}, as long as $X$ is the base change of a scheme over $\mb{Z}$.

\begin{thm}\label{thm:enriched gottsche surface}
Let $k$ be a field of characteristic not 2. Let $X\to\Spec\mb{Z}$ be a smooth, separated, and projective scheme of relative dimension 2. Let $e_\mb{C}:=\chi(X(\mb{C}))$ and $e_\mb{R}:=\chi(X(\mb{R}))$. Then
\begin{equation}\label{eq:enriched gottsche surface}
\sum_{g\geq 0}\chi^c(\Hilb^g X_k)\cdot t^g=\prod_{r\geq 1}(1-\langle-1\rangle^r\cdot t^r)^{-e_\mb{R}}\cdot\prod_{s\geq 1}(1-t^{2s})^{-\frac{e_\mb{C}-e_\mb{R}}{2}}
\end{equation}
in $\GW(k)\dlb t\drb$.
\end{thm}
\begin{proof}
The Hilbert scheme of a projective scheme is again projective, and $\Hilb^g(X)$ is also smooth by \cite[Proposition 7.27]{ES87}. We have $\Hilb^gX\times_\mb{Z}\Spec{k}\cong\Hilb^g X_k$ by \cite[p.~112]{FGA}, so \cite[Theorem 5.11]{BachmannWickelgren} implies that 
\begin{equation}\label{eq:chi^c of Hilb}
\chi^c(\Hilb^g X_k)=\frac{n_\mb{C}+n_\mb{R}}{2}\langle 1\rangle+\frac{n_\mb{C}-n_\mb{R}}{2}\langle -1\rangle,
\end{equation}
where $n^g_\mb{C}$ and $n^g_\mb{R}$ are the Euler characteristics of the complex and real loci of $\Hilb^g X$, respectively. These Euler characteristics fit into the generating series
\begin{align*}
    \sum_{g\geq 0}n^g_\mb{C}\cdot t^g&=\prod_{n\geq 1}(1-t^n)^{-e_\mb{C}},\\
    \sum_{g\geq 0}n^g_\mb{R}\cdot t^g&=\prod_{r\geq 1}(1-(-t)^r)^{-e_\mb{R}}\cdot\prod_{s\geq 1}(1-t^{2s})^{-\frac{e_\mb{C}-e_\mb{R}}{2}}
\end{align*}
by \cite{Got90} and \cite[Theorem B]{KR15}, respectively. Taking the rank and signature of Equation~\ref{eq:enriched gottsche surface} recovers these two generating series, so Equation~\ref{eq:enriched gottsche surface} is the generating series for Equation~\ref{eq:chi^c of Hilb}.
\end{proof}

Theorem~\ref{thm:enriched gottsche surface} lends some credence to Conjecture~\ref{conj:main}, but it will not serve our goal of enriching the Yau--Zaslow formula. This is because no K3 surface is smooth and proper over $\mb{Z}$ \cite{Fon93}.

\subsection{Motivic Euler characteristic of K3 surfaces}
Given a specific K3 surface $X$, one might wish to compute the motivic Euler characteristic $\chi^c(X)$ and simplify the product $\prod_{n\geq 1}(1-\langle-1\rangle^{n-1}t^n)^{-\chi^c(X)}$ in terms of the conjectural power structure on $\GW(k)$.

Let $X$ be a K3 surface over a field $k$ of characteristic not 2. Since every K3 surface over a field is smooth and proper, Proposition~\ref{prop:compact=cat} and Theorem~\ref{thm:LevineRaksit} imply that $\chi^c(X)=\mb{H}+P_\dR(X)$, where
\[P_\dR(X):H^1(X,\Omega^1_X)\otimes H^1(X,\Omega^1_X)\xrightarrow{\smile}H^2(X,\omega_X)\xrightarrow{\eta}k\]
is (the isomorphism class of) the bilinear form given by composing the intersection pairing with the trace map from coherent duality. Since $\omega_X\cong\mc{O}_X$, Serre duality implies that $H^2(X,\omega_X)\cong H^2(X,\mc{O}_X)\cong H^0(X,\mc{O}_X)$. Now $\dim H^0(X,\mc{O}_X)=1$ for all K3 surfaces by definition, so the trace map $\eta:H^2(X,\omega_X)\to k$, being non-degenerate, must be an isomorphism. In particular, $P_\dR(X)$ is entirely determined by the cup product on $H^1(X,\Omega^1_X)$.

\begin{rem}
Since $\sign\mb{H}=0$, Proposition~\ref{prop:sign = etale of real locus} implies that $\chi(X(\mb{R}))=\sign{P_\dR(X)}$ for any real K3 surface $X$. The Euler characteristic of the real locus of a K3 surface attains any even value between $-18$ and 20 (as mentioned in \cite[\S 1]{KR17}; see \cite[\S 3.7.2]{DK00} for details).
\end{rem}

The remainder of this section is dedicated to sharing a loose end that we could not tie off. A classical invariant of a K3 surface $X$ is the Picard lattice \cite[Chapter 1, Section 2]{Huy16}, which is the finitely generated free abelian group $\op{Pic}(X)$, together with the intersection form
\[P_\mr{Pic}(X):\op{Pic}(X)\times\op{Pic}(X)\to\mb{Z}.\]
Over $\mb{C}$, the cohomology group $H^2(X(\mb{C}), \mb{Z})$ is known to be the rank 22 lattice $\Lambda_{K3}\colon = E_8(-1)^{\oplus 2}\oplus\mb{H}^{\oplus 3}$ \cite[Chapter 1, Proposition 3.5]{Huy16}, where
\[E_8(-1)=\begin{pmatrix}
    -2 &  1 &    &    &    &    &    &\\
     1 & -2 &  1 &    &    &    &    &\\
       &  1 & -2 &  1 &  1 &    &    &\\
       &    &  1 & -2 &  0 &    &    &\\
       &    &  1 &  0 & -2 & 1  &    &\\
       &    &    &    &  1 & -2 &  1 &\\
       &    &    &    &    &  1 & -2 &  1\\
       &    &    &    &    &    &  1 & -2\\
\end{pmatrix}.\]
The Picard lattice is the sublattice $H^{1,1}(X) \cap H^2(X(\mb{C}), \mb{Z})$ by the Lefschetz $(1,1)$-theorem. By the Hodge decomposition, this is uniquely determined by the line $H^{0,2} = H^2(X, \mathcal{O}_X) \subset H^2(X(\mb{C}), \mb{C}) \cong \Lambda_{K3}\otimes_{\mb{Z}} \mb{C}$. Over other fields, determining the rank $\rho(X)$ of the Picard lattice is a difficult question. Over $\mb{R}$, the Hodge index theorem implies that $\sign{P_\mr{Pic}(X)}=(1,\rho(X)-1)$.

As the cup product is compatible with the intersection product, it is natural to ask how $P_\dR(X)\in\GW(k)$ and $P_\mr{Pic}(X)\in\GW(\mb{Z})$ relate. Note that $P_\dR(X)$ is not simply the base change of $P_\mr{Pic}(X)$, as $\rank{P_\dR(X)}=22$, while $\rank{P_\mr{Pic}(X)}=\rho(X)$ can vary. Similarly, $\sign{P_\mr{Pic}(X)}=2-\rho(X)\leq 0$ (as $\rho(X)\geq 1$), whereas $\sign{P_\dR(X)}$ can be positive.

Instead, we could try looking at the cup product on $\tau$-cohomology for $\tau\in\{\mr{Zar},\Nis,\et\}$ \cite[\S8.4]{Jar15}. Hilbert's Theorem 90 implies that $H^1_\tau(X,\mb{G}_m)\cong\op{Pic}(X)$ \cite[\href{https://stacks.math.columbia.edu/tag/03P8}{Theorem 03P8}]{stacks}. The cup product is a map
\[\smile\, :H^1_\tau(X,\mb{G}_m)\otimes H^1_\tau(X,\mb{G}_m)\to H^2_\tau(X,\mb{G}_m\otimes\mb{G}_m)\cong H^2_\tau(X,\mb{G}_m).\]
Now our choice of $\tau$ becomes relevant. Since K3 surfaces are noetherian, integral, and locally factorial, $H^i_\Zar(X,\mb{G}_m)=0$ for all $i>1$. On the other hand, $H^2_\et(X,\mb{G}_m)$ is interesting. Gabber proved that $H^2_\et(X,\mb{G}_m)$ is torsion if $X$ admits an ample invertible sheaf, identifying $H^2_\et(X,\mb{G}_m)$ with the cohomological Brauer group $\op{Br}(X):=H^2_\et(X,\mb{G}_m)_\tors$. If we had a ``$\tau$-Brauer trace form'' $\eta_\tau:H^2_\tau(X,\mb{G}_m)\to k$, we could define a bilinear form
\[P_\tau(X):H^1_\tau(X,\mb{G}_m)\otimes H^1_\tau(X,\mb{G}_m)\xrightarrow{\smile} H^2_\tau(X,\mb{G}_m)\xrightarrow{\eta_\tau} k\]
and try comparing the classes $P_\tau(X),P_\dR(X)\in\GW(k)$. For example, if we have a rational point $x\in X(k)$, then we could try defining $\eta_\tau$ as the composite
\[\begin{tikzcd}[row sep=0em]
    \op{Br}(X)\arrow[r,"x_*"] & \op{Br}(k)\arrow[r,"\op{norm}"] & k\\
    \alpha\arrow[r,mapsto] & \alpha(x)\arrow[r,mapsto] & \op{norm}(\alpha(x)),
\end{tikzcd}\]
but the potential dependence on $x$ would be undesirable.

\section{Proving Conjecture~\ref{conj:main} over pythagorean fields}\label{sec:pythagorean}
To conclude this article, we will show that if $X$ is a quasi-projective variety over a field $k$ of characteristic not 2, and if $\chi^c(X)=0$, then $\chi^c(\Sym^nX)$ is torsion for all $n\geq 1$.

\begin{prop}\label{prop:torsion}
Let $k$ be a field of characteristic not 2. Let $f:X\to Y$ be a morphism of $k$-varieties. If $\rank\chi^c(X)=0$ and $\sign_\sigma\chi^c(X)=0$ with respect to all real closed embeddings $\sigma$, then $\chi^c(X)\in\GW(k)_\tors$.
\end{prop}
\begin{proof}
This is essentially Pfister's local-global principle \cite[Theorem 2.7.3]{Sch85}, which characterizes the torsion subgroup of the Witt group $\mr{W}(k)$ as the kernel of the signature map. Since the group $\mb{Z}$ is torsion-free, $\GW(k)_\tors$ must lie in the kernel of $\rank:\GW(k)\to\mb{Z}$. We then conclude by recalling that $\mr{W}(k)\cong\GW(k)/(\mb{H})$, so that 
\[\GW(k)_\tors=\ker\rank\cap\bigcap_{\substack{\sigma:k\to R\\ R\text{ real closed}}}\ker\sign_\sigma.\qedhere\]
\end{proof}

\begin{cor}\label{cor:symmetric torsion}
Let $k$ be a field of characteristic not 2. Let $X$ be a quasi-projective variety over $k$. If $\chi^c(X)\in\GW(k)_\tors$, then $\chi^c(\Sym^nX)\in\GW(k)_\tors$ for all $n\geq 1$.
\end{cor}
\begin{proof}
The assumption $\chi^c(X)\in\GW(k)_\tors$ implies that $\chi^\et(X)=0$ (by Proposition~\ref{prop:rank = etale}) and $\chi^t(X_\sigma(R))=0$ for any real closed embedding $\sigma:k\to R$ (by Proposition~\ref{prop:sign = etale of real locus}). 

To show that $\rank\chi^c(\Sym^n X)=0$, it suffices to prove that $\chi^\et(\Sym^n_\Delta X ) = 0$. Let $C_m(X)$ denote $X^m \setminus \Delta$ where $\Delta$ is the big diagonal. Then $\Sym^n_\Delta X$ is a disjoint union of quotients of $C_m(X)$ by free group actions for various $m$. Since $\chi^{\et}$ is multiplicative for \'etale maps, it suffices to show $\chi^{\et}(C_m(X)) = 0$. Now we induct on $m$. For $m = 1$, we have $\chi^{\et}(X) = 0$ by assumption. Now suppose $\chi(C_n(X)) = 0$ for all $n \le m$. The projection $C_{m+1}(X) \to C_m(X)$ forgetting the last coordinate is an \'etale-locally trivial fibration over a base with $\chi^{\et}(C_m(X)) = 0$, and so $\chi^{\et}(C_{m+1}(X)) = 0$. 

It remains to show that if $\sigma:k\to R$ is a real closed embedding, then $\sign_\sigma\chi^c(\Sym^n X)=0$ for all $n\geq 1$. By Proposition~\ref{prop:sign = etale of real locus}, for any $k$-variety $Y$, we have
\[\sign_\sigma\chi^c(\Sym^n X)=\chi^t((\Sym^n X)_\sigma(R)),\]
where $\chi^t$ is the compactly supported Euler characteristic and $(\Sym^n X)_\sigma(R)$ is the topological space of $R$-points of $\Sym^n X$ under the embedding $\sigma$. (The topology on $R$ is the one induced by the unique ordering on $R$, generally referred to as the strong topology.) When $R=\mb{R}$, \cite[Equation (3.1)]{KR15} gives us
\[(\Sym^n X)_\sigma(\mb{R})=\coprod_{a+2b=n}\Sym^a(X_\sigma(\mb{R}))\times\Sym^b(X'_\mb{C}),\]
where $X'_\mb{C}:=\frac{X_{\overline{\sigma}}(\mb{C})-X_\sigma(\mb{R})}{\text{conjugation}}$ and $X_{\overline{\sigma}}(\mb{C})$ is the topological space underlying the complex points of $X$. For an arbitrary real closed field $R$, we analogously have
\[(\Sym^n X)_\sigma(R)=\coprod_{a+2b=n}\Sym^a(X_\sigma(R))\times\Sym^b(X'_{\overline{R}}),\]
where $\overline{R}:=R(\sqrt{-1})$ and $X'_{\overline{R}}:=\frac{X_{\overline{\sigma}}(\overline{R})-X_\sigma(R)}{\text{conjugation}}$. Here, $X_{\overline{\sigma}}(\overline{R})$ signifies the topological space underlying the geometric points $X(\overline{R})$, where the topology is induced by endowing the set $\overline{R}=R^2$ with the product topology. Thus
\[\chi^t((\Sym^n X)_\sigma(R))=\sum_{a+2b=n}\chi^t(\Sym^a(X)_\sigma(R))\cdot\chi^t(\Sym^b(X'_{\overline{R}})).\]
It is a classical computation that
\begin{equation}\label{eq:chi binomial}
\chi^t(\Sym^n M)=\binom{\chi^t(M)+n-1}{n}
\end{equation}
for any topological space $M$. This is done by stratifying $\Sym^n M$ so that the quotient $M^n/S_n$ is given by a finite free action over each stratum (see e.g.~\cite{Mac62}). Since $\chi^c(X)=0$, we have $\chi^t(X_\sigma(R))=0$ and hence $\chi^t(\Sym^a(X_\sigma(R)))=0$ for all $a\geq 1$. It follows that
\[\chi^t((\Sym^n X)_\sigma(R))=\begin{cases} 0 & n\text{ odd},\\
\chi^t(\Sym^{n/2}(X'_{\overline{R}})) & n\text{ even}.
\end{cases}\]
It remains to prove that $\chi^t(\Sym^{n/2}(X'_{\overline{R}}))=0$. Over an algebraically closed field, the \'etale Euler characteristic and \'etale Euler characteristic with compact support coincide \cite[(SGA 5) Expos\'e X]{SGA5} (see also \cite{Lau81}). Moreover, $\chi^t(X_{\overline{\sigma}}(\overline{R}))=\chi^\et(X)$ by \cite{Hub84} (see also \cite[\S 15]{Sch94}).  Since $\chi^c(X)=0$, Proposition~\ref{prop:rank = etale} now implies that $\rank\chi^c(X)=\chi^\et(X)=\chi^t(X_{\overline{\sigma}}(\overline{R}))=0$. ``Complex'' conjugation, induced by
\begin{align*}
    R(\sqrt{-1})&\to R(\sqrt{-1})\\
    \sqrt{-1}&\mapsto -\sqrt{-1},
\end{align*}
is a free $\mb{Z}/2\mb{Z}$-action on $X_{\overline{\sigma}}(\overline{R})-X_\sigma(R)$, so $\chi^t(X'_{\overline{R}})=\chi^t(X_{\overline{\sigma}}(\overline{R})-X_\sigma(R))/2=0$. The desired result now follows from Equation~\ref{eq:chi binomial}.
\end{proof}

As a corollary, we also obtain the following compatibility of $\chi^{\et}$ and $\chi^t$ with power structures. 

\begin{thm}\label{thm:complex and real power structures} 
Let $k$ be a field of characteristic not $2$. 
\begin{enumerate}[(i)] 
\item The ring homomorphism
$$
\chi^{\et} = \rank \chi^c : K_0(\rm{Var}_k) \to \mathbb{Z}
$$
is compatible with the natural power structures. 
\item Let $\sigma : k \to R$ be a real closed embedding. Then the ring homomorphism
$$
\sign_\sigma \chi^c : K_0(\rm{Var}_k) \to \mathbb{Z}
$$
is compatible with the natural power structures. 
\end{enumerate}
\end{thm}
\begin{proof} Recall the natural power structure on $\mathbb{Z}$ is given by exponentiation of power series. Let $\varphi$ denote either $\chi^\et$ or $\sign_\sigma\chi^c$. By Lemma \ref{lem:induced power structure} and the proof of Proposition \ref{prop:kernel of chi^c}, it suffices to check that $\varphi(\Sym^n X) = 0$ for all $X$ with $\varphi(X) = 0$. This is precisely Corollary \ref{cor:symmetric torsion}.
\end{proof}

\begin{rem}
Rather than asking for Conjecture~\ref{conj:main} to hold for all elements of $K_0(\Var_k)$, one can instead consider the largest subring of $K_0(\Var_k)$ on which the values of $\chi^c$ are compatible with the expected power structure on $\GW(k)$. This question is explored in \cite{PRV24}, where the authors dub this subring $K_0(\Sym_k)$. They show that classes of cellular varieties belong to this subring by applying a theorem of G\"ottsche.

Theorem~\ref{thm:complex and real power structures} implies that if $\chi^c(X)$ is determined by its rank and signature (e.g.~if $\chi^c(X)\in\mb{Z}\cdot\{\langle 1\rangle,\langle -1\rangle\}$), then $X\in\Sym_k$. This includes all cellular varieties by Lemma~\ref{lem:cellular decomposition}, giving an alternate proof that cellular varieties are ``symmetrizable.'' By \cite[Theorem~5.11]{BachmannWickelgren}, any variety that is the base change of a smooth proper $\mb{Z}$-scheme also satisfies these criteria. This includes projective spaces, Grassmannians, and Hilbert schemes of smooth projective curves and surfaces, although all of these examples are already cellular.
\end{rem}

Finally, Corollary~\ref{cor:symmetric torsion} implies that Conjecture~\ref{conj:main} is true over any field for which $\GW(k)$ is torsion-free.

\begin{thm}\label{thm:conj over pythagorean}
Let $k$ be a field of characteristic not 2 such that $\GW(k)_\tors=\{0\}$ (such as a pythagorean field of characteristic not 2). Then Conjecture~\ref{conj:main} is true over $k$.
\end{thm}
\begin{proof}
Let $X$ be a quasi-projective variety over $k$. If $\chi^c(X)=0$, then Corollary~\ref{cor:symmetric torsion} implies that $\chi^c(\Sym^nX)\in\GW(k)_\tors$ for all $n\geq 1$. By assumption, $\GW(k)_\tors=\{0\}$, giving the desired result.   
\end{proof}

It follows that all of the conditional results presented in this article hold over such fields, which are precisely those fields with 2-primary virtual cohomological dimension at most 1 (by a special case, proven by Merkur'ev \cite[Theorem 2.2]{Mer81}, of the Milnor conjecture). 

\bibliography{yz}{}
\bibliographystyle{alpha}
\end{document}